\documentclass[reqno,11pt]{amsart}

\usepackage[a4paper,left=23mm,right=23mm,top=30mm,bottom=30mm,marginpar=25mm]{geometry}
\usepackage{amsmath}
\usepackage{amssymb}
\usepackage{amsthm}
\usepackage{amscd}
\usepackage{mathtools}
\usepackage{url}
\usepackage[T1]{fontenc}

\usepackage{color}
\usepackage[final]{graphicx}
\usepackage{enumitem}
\usepackage{esint}
\usepackage{dsfont}
\usepackage{tikz}
\usetikzlibrary{arrows.meta}
\usepackage{cite}
\usepackage{bbm}
\usepackage[colorlinks=true, linkcolor=blue,  anchorcolor=blue,citecolor=blue]{hyperref}
\usepackage[english]{babel}
\usepackage{subfig,graphicx}

\usepackage[only,llbracket,rrbracket]{stmaryrd}

\numberwithin{equation}{section}

\newtheoremstyle{thmlemcorr}{10pt}{10pt}{\itshape}{}{\bfseries}{.}{10pt}{{\thmname{#1}\thmnumber{ #2}\thmnote{ (#3)}}}
\newtheoremstyle{thmlemcorr*}{10pt}{10pt}{\itshape}{}{\bfseries}{.}\newline{{\thmname{#1}\thmnumber{ #2}\thmnote{ (#3)}}}
\newtheoremstyle{defi}{10pt}{10pt}{\itshape}{}{\bfseries}{.}{10pt}{{\thmname{#1}\thmnumber{ #2}\thmnote{ (#3)}}}
\newtheoremstyle{remexample}{10pt}{10pt}{}{}{\bfseries}{.}{10pt}{{\thmname{#1}\thmnumber{ #2}\thmnote{ (#3)}}}
\newtheoremstyle{ass}{10pt}{10pt}{}{}{\bfseries}{.}{10pt}{{\thmname{#1}\thmnumber{ A#2}\thmnote{ (#3)}}}

\theoremstyle{thmlemcorr}
\newtheorem{theorem}{Theorem}
\numberwithin{theorem}{section}
\newtheorem{lemma}[theorem]{Lemma}

\newtheorem{proposition}[theorem]{Proposition}

\theoremstyle{thmlemcorr*}
\newtheorem{theorem*}{Theorem}
\newtheorem{lemma*}[theorem]{Lemma}
\newtheorem{corollary*}[theorem]{Corollary}
\newtheorem{proposition*}[theorem]{Proposition}
\newtheorem{problem*}[theorem]{Problem}
\newtheorem{conjecture*}[theorem]{Conjecture}

\theoremstyle{defi}
\newtheorem{definition}[theorem]{Definition}

\theoremstyle{remexample}
\newtheorem{remarkx}[theorem]{Remark}
\newenvironment{remark}
  {\pushQED{\qed}\remarkx}
  {\popQED\endremarkx}
\newtheorem{examplex}[theorem]{Example}
\newenvironment{example}
  {\pushQED{\qed}\examplex}
  {\popQED\endexamplex}

\theoremstyle{ass}

\newcommand{\N}{\mathbb{N}}
\newcommand{\R}{\mathbb{R}}
\newcommand{\C}{\mathbb{C}}

\newcommand{\Z}{\mathbb{Z}}

\renewcommand{\S}{\mathbb{S}}

\newcommand{\Acal}{\mathcal{A}}
\newcommand{\Bcal}{\mathcal{B}}

\newcommand{\Dcal}{\mathcal{D}}
\newcommand{\Ecal}{\mathcal{E}}
\newcommand{\Fcal}{\mathcal{F}}

\newcommand{\Hcal}{\mathcal{H}}
\newcommand{\Ical}{\mathcal{I}}
\newcommand{\Jcal}{\mathcal{J}}

\newcommand{\Pcal}{\mathcal{P}}

\newcommand{\Tcal}{\mathcal{T}}
\newcommand{\Ucal}{\mathcal{U}}

\newcommand{\Tc}{\mathcal{T}}

\renewcommand{\epsilon}{\varepsilon}
\newcommand{\dd}{\,\mathrm{d}}

\DeclarePairedDelimiter{\abs}{\lvert}{\rvert}
\DeclarePairedDelimiter{\norm}{\lVert}{\rVert}

\DeclareMathOperator{\dist}{dist}
\DeclareMathOperator*{\argmin}{arg\,min}
\DeclareMathOperator*{\argmax}{arg\,max}

\DeclareMathOperator{\Div}{div}

\DeclareMathOperator{\Tr}{Tr}
\renewcommand{\O}{\Omega}
\definecolor{dgreen}{rgb}{0,0.5,0}

\renewcommand{\d}{\mathrm{d}}

\def\XXint#1#2#3{{\setbox0=\hbox{$#1{#2#3}{\int}$}
		\vcenter{\hbox{$#2#3$}}\kern-.5\wd0}}

\newcommand{\mres}{\mathbin{\vrule height 1.4ex depth 0pt width
0.13ex\vrule height 0.13ex depth 0pt width 1.0ex}}

\newcommand\restr[2]{{
  \left.\kern-\nulldelimiterspace
  #1 
  \vphantom{\big|} 
  \right|_{#2}
  }}

\title[Shear band formation through discrete plastic increments]{Formation of shear bands through discrete localized plastic increments}

\author{Heiner Olbermann}
\address{Research Institute in Mathematics and Physics, Universit\'{e} catholique de Louvain, Chemin du Cyclotron 2, 1348 Louvain-la-Neuve, Belgium}
\email{heiner.olbermann@uclouvain.be}

\author{Hidde Sch\"{o}nberger}
\address{Department Applied Mathematics, University of Twente, Drienerlolaan 5, 7522 NB Enschede, The Netherlands}
\email{hidde.schonberger@utwente.nl}

\begin{document}
\allowdisplaybreaks
\maketitle
\begin{abstract}
We study rate-independent systems in which the evolution is obtained by incremental minimization of elastic energy and one-homogeneous dissipation under constraints on the admissible inelastic increments. Such constraints arise naturally in several models of inelastic materials, including shear-transformation-zone models, intermittent plasticity, and quantized plasticity, where evolution proceeds through prescribed elementary activation events rather than arbitrary plastic increments. We establish the existence of a continuous-time limit for the corresponding constrained incremental scheme under suitable assumptions.

As an application, we formulate a simplified shear-transformation-zone model consisting of localized activation events on a lattice embedded in a two-dimensional elastic medium. We prove that this model generates shear-band formation through the successive activation of neighboring transformation zones. More precisely, we show that every admissible evolution remains localized on a growing line segment in the infinite-domain setting and that the same mechanism persists on sufficiently large bounded domains. These results provide a rigorous mathematical framework for constrained activation models and show that restrictions on admissible elementary events alone constitute a mechanism for localization in rate-independent systems.    \medskip
    
    \medskip


\end{abstract}

\thispagestyle{empty}

\section{Introduction}

Rate-independent systems are frequently modeled through incremental minimization of an en\-ergy-dissipation functional \cite{MiR15,mielke2003energetic}. This framework provides a robust variational description of quasistatic evolution and has found numerous applications in plasticity, fracture, damage, and phase transformations. At the same time, the one-homogeneous dissipation characteristic of rate-indepen\-dent processes introduces a substantial degree of degeneracy into the evolution. As a consequence, a given loading process may admit multiple admissible evolutions, reflecting the absence of an intrinsic time scale and the lack of strict convexity in the inelastic variables. A central question is therefore how physically relevant evolutions are selected among the many mathematically admissible ones.

In the present work, we investigate a class of selection mechanisms based on restrictions of the admissible inelastic increments. More precisely, we consider incremental minimization schemes in which the inelastic increment is not allowed to vary freely but is constrained to belong to a prescribed subset of the unconstrained inelastic strain space. Such restrictions arise naturally in a variety of physical theories in which inelastic evolution is assumed to proceed through elementary events rather than arbitrary deformations. The resulting dynamics may be viewed as a constrained rate-independent evolution whose qualitative properties can differ substantially from those of the corresponding unconstrained model.

Our primary objective is to understand the mathematical consequences of such constraints. In particular, we are interested in whether they can serve as a mechanism for the spontaneous emergence of localized structures. Localization phenomena play a fundamental role in many nonlinear systems and are often attributed to softening, loss of ellipticity, nonconvex constitutive behavior, or other instability mechanisms. The framework considered here suggests a different possibility: localization generated by restrictions on the admissible evolution itself. The question we address is whether lower-dimensional structures may arise solely as a consequence of constrained incremental dynamics.

\medskip

\subsection*{Contributions of the paper} In a first result of the present paper (see Theorem~\ref{thm:existence_rimp}), we establish the existence of a continuous-time evolution associated with constrained incremental minimization. Under the assumption that, apart from the trivial increment $0$, every admissible inelastic increment $q$ has dissipation bounded from below by a constant (see \eqref{eq:Idiscrete}), we consider discrete solutions obtained by successive minimization of elastic energy plus one-homogeneous dissipation. We prove that, as the time step tends to zero, these discrete solutions converge to a limiting rate-independent evolution, which we call a \emph{constrained energetic solution} (see~Definition~\ref{def:ces}). The inelastic part of such a solution is piecewise constant in time. Each of its jumps may be interpreted as the macroscopic trace of a finite cascade of elementary inelastic transitions through successive metastable states.

The constrained evolution naturally leads to the study of the metastable states that can be reached from a given reference configuration through successive admissible minimizing increments. We consider a single model in which each admissible plastic increment is a localized elementary activation supported at one site of a prescribed lattice. By varying the geometry of the domain and the choice of lattice, we obtain a series of results showing that the reachable metastable configurations are organized along one-dimensional structures.

For an infinite body subject to uniform shear, we first show that, starting from vanishing plastic deformation, successive admissible minimizing increments align along a one-dimensional chain of activation sites, thereby producing a shear-band-type structure (Theorem \ref{th:shearband}). In a second whole-space result, we compare the behavior of activation bands on square and triangular lattices. We prove that a horizontal band on the square lattice is completely stable (Proposition \ref{prop:no_fattening_quad_lattice}), whereas on the triangular lattice we characterize all metastable configurations accessible from a pre-existing band and show that propagation can occur only through the addition of aligned rows at its two sides (Theorem \ref{thm:thickening_bands}). In the half-plane, we obtain an analogous characterization in which the localized structure originates at the boundary and propagates into the interior (Theorem \ref{th:shearneumann}). For bounded domains, we prove that the same localization mechanisms persist up to distances proportional to the distance of the initial activation from the boundary: before boundary effects become significant, every reachable metastable configuration remains supported on a one-dimensional chain of activation sites (Theorems \ref{th:partialshearband} and \ref{thm:bounded_neumann}). Taken together, these results establish both the formation and the subsequent propagation of localized plastic structures in several geometries, see Figure~\ref{fig} for a visual overview.

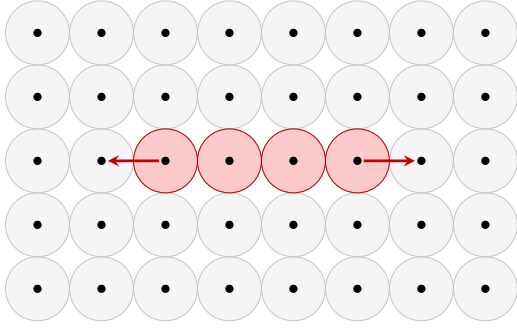
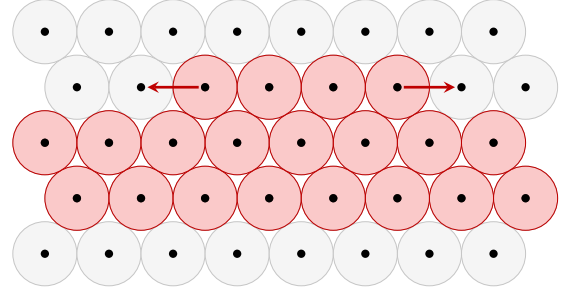
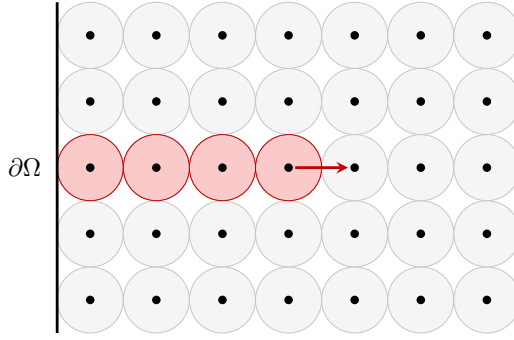
\begin{figure}\centering
\subfloat[Formation of shear band in whole-space (Theorem~\ref{th:shearband})]{\label{a}\resizebox{0.42\textwidth}{!}{
\begin{tikzpicture}[
    lattice/.style={gray!35, thin},
    nodept/.style={circle, fill=black, inner sep=1.3pt},
    zone/.style={circle, draw=gray!45, fill=gray!8, 
                 minimum size=10mm, inner sep=0pt},
    activezone/.style={circle, draw=red!75!black, fill=red!35,
                      fill opacity=0.55, draw opacity=1,
                      minimum size=10mm, inner sep=0pt},
    growth/.style={->, >=stealth, very thick, red!75!black}
]

\def\Nx{7}
\def\Ny{4}

\foreach \x in {0,...,\Nx}{
    \draw[lattice] (\x,0) -- (\x,\Ny);
}
\foreach \y in {0,...,\Ny}{
    \draw[lattice] (0,\y) -- (\Nx,\y);
}

\foreach \x in {0,...,\Nx}{
    \foreach \y in {0,...,\Ny}{
        \node[zone] at (\x,\y) {};
    }
}

\foreach \x in {2,3,4,5}{
    \node[activezone] at (\x,2) {};
}

\foreach \x in {0,...,\Nx}{
    \foreach \y in {0,...,\Ny}{
        \node[nodept] at (\x,\y) {};
    }
}

\draw[growth] (1.9,2) -- (1.1,2);
\draw[growth] (5.1,2) -- (5.9,2);


\end{tikzpicture}}}\hfill
\subfloat[Thickening of shear band in whole-space (Theorem~\ref{thm:thickening_bands})]{\label{b}\resizebox{0.44\textwidth}{!}{\begin{tikzpicture}[
    scale=1.0,
    site/.style={
        circle,
        fill=black,
        inner sep=1.3pt
    },
    zone/.style={
        circle,
        draw=gray!45, fill=gray!8,
        thin,
        minimum size=1cm
    },
    active/.style={
        circle,
        draw=red!75!black, fill=red!35,
        fill opacity=0.55, draw opacity=1,
        minimum size=10mm, inner sep=0pt
    },
	growth/.style={->, >=stealth, very thick, red!75!black}
]

\def\Nx{7}
\def\Ny{4}

\foreach \j in {0,...,\Ny} {
    \foreach \i in {0,...,\Nx} {

        \pgfmathsetmacro{\x}{\i + 0.5*mod(\j,2)}
        \pgfmathsetmacro{\y}{0.8660254*\j}

        \node[zone] at (\x,\y) {};

        \node[site] at (\x,\y) {};
    }
}

\foreach \j in {1,2} {
    \foreach \i in {0,...,\Nx} {
        \pgfmathsetmacro{\x}{\i + 0.5*mod(\j,2)}
        \pgfmathsetmacro{\y}{0.8660254*\j}

        \node[active] at (\x,\y) {};
        \node[site]   at (\x,\y) {};
    }
}

\def\jpartial{3}

\foreach \i in {2,3,4,5} {
    \pgfmathsetmacro{\x}{\i + 0.5*mod(\jpartial,2)}
    \pgfmathsetmacro{\y}{0.8660254*\jpartial}

    \node[active] at (\x,\y) {};
    \node[site]   at (\x,\y) {};
}

\pgfmathsetmacro{\yp}{0.8660254*\jpartial}

\draw[growth]
    (2.4,\yp) -- (1.6,\yp);

\draw[growth]
    (5.6,\yp) -- (6.4,\yp);

\end{tikzpicture}}}\par 
\subfloat[Formation of shear band in half-space (Theorem~\ref{th:shearneumann})]{\label{c}\resizebox{0.42\textwidth}{!}{\begin{tikzpicture}[
    lattice/.style={gray!35, thin},
    boundary/.style={black, very thick},
    nodept/.style={circle, fill=black, inner sep=1.3pt},
    zone/.style={circle, draw=gray!45, fill=gray!8,
                 minimum size=10mm, inner sep=0pt},
    activezone/.style={circle, draw=red!75!black, fill=red!35,
                      fill opacity=0.55, draw opacity=1,
                      minimum size=10mm, inner sep=0pt},
    growth/.style={->, >=stealth, very thick, red!75!black}
]

\def\Nx{6}
\def\Ny{4}


\foreach \x in {0,...,\Nx}{
    \draw[lattice] (\x,0) -- (\x,\Ny);
}

\foreach \y in {0,...,\Ny}{
    \draw[lattice] (0,\y) -- (\Nx,\y);
}

\draw[boundary] (0,-0.5) -- (0,\Ny+0.5);


\foreach \x in {0,...,\Nx}{
    \foreach \y in {0,...,\Ny}{
        \node[zone] at (\x,\y) {};
    }
}


\foreach \x in {0,1,2,3}{
    \node[activezone] at (\x,2) {};
}


\foreach \x in {0,...,\Nx}{
    \foreach \y in {0,...,\Ny}{
        \node[nodept] at (\x,\y) {};
    }
}


\draw[growth] (3.1,2) -- (3.9,2);

\draw[boundary] (-0.5,-0.5) -- (-0.5,\Ny+0.5);

\node[left=3pt] at (-0.5,\Ny/2)
    {$\partial \Omega$};

\end{tikzpicture}}}
\caption{Schematic representation of the activation of localized plastic zones in anti-plane shear with various geometries subject to an affine boundary condition $g(x)=\gamma x_2$ at infinity and Neumann boundary conditions at $\partial \Omega$ for the half-space.}
\label{fig}
\end{figure}

\medskip

\subsection*{Relation to literature} The model and all of the above localization results admit an interpretation within shear transformation zone (STZ) theory. Localized shear transformations have long been regarded as elementary carriers of plastic deformation in amorphous solids, beginning with the pioneering work of Spaepen \cite{spaepen1977microscopic} and Argon \cite{argon1979plastic}, and with the later mesoscopic descriptions of Bulatov and Argon \cite{Bul94a,Bul94b,Bul94c}. Building on these ideas, Falk and Langer \cite{falk1998dynamics,falk2011deformation} developed STZ theory as a continuum constitutive framework in which plastic flow is generated by populations of localized rearrangement events. Among the subsequent mesoscopic formulations, the model of Homer and Schuh \cite{HoS09} represents plastic evolution through the activation of individual STZ sites embedded in an elastic medium. After passing to a zero-temperature limit, replacing finite-element elasticity by continuous elasticity, distributing the activation sites on a lattice, and restricting the elastic response to anti-plane shear, this model falls within the framework considered here. The localization theorems therefore provide a rigorous mechanism for the formation and propagation of shear-band-type structures in this simplified STZ setting. At the same time, the dichotomy between the results on the square and triangular lattice show that this model has a strong dependency on the location of the zones, which may translate to mesh dependence in the finite element simulations. Moreover, the indefinite thickening of the shear band for the case of the triangular lattice (Theorem~\ref{thm:thickening_bands}) indicates that this model may not give rise to an intrinsic length-scale for the thickness of the bands, which warrants further investigation into this aspect of the physical models.

More generally, the framework considered here may be viewed as an abstraction of a recurring modeling principle in the theory of inelastic materials. In crystalline plasticity, plastic flow is restricted by crystallographic slip systems and compatibility conditions. In STZ models for amorphous solids, plastic deformation is represented by localized rearrangement events. In mesoscopic tensorial models and theories of quantized plasticity \cite{salman2011minimal,salman2012critical,zhang2020variety,perchikov2024quantized}, evolution proceeds through elementary transformations between discrete states. Despite their different physical origins, these theories share the common feature that admissible inelastic increments are restricted to a distinguished subset of the full inelastic strain space. These models are based on the physical hypothesis that inelastic deformation proceeds through elementary localized events rather than arbitrary infinitesimal increments of the inelastic strain. The present work isolates this common mathematical structure and investigates its consequences for the resulting rate-independent dynamics. Our main result shows that restricting admissible elementary events constitutes an independent mechanism for selecting rate-independent evolutions and generating localization.

A related perspective arises in the theory of balanced-viscosity solutions for rate-independent systems \cite{MiR15}. Both frameworks may be viewed as mechanisms for selecting among multiple admissible evolutions. In balanced-viscosity theory, the selection originates from a vanishing-viscosity limit, whereas in the present setting it is imposed directly through constraints on admissible increments. Although we do not pursue this connection in a systematic manner, it provides useful intuition for the role of constraints as a selection principle in degenerate rate-independent evolution.

\medskip

\subsection*{Outline} The paper is organized as follows. In Section \ref{sec:setting} we introduce the general framework of constrained incremental minimization and establish the existence of constrained energetic solutions. We then turn to the analysis of the simplified STZ model in dimension 2 (anti-plane shear), and the emergence of one-dimensional structures therein: Section~\ref{sec:localized_plastic} details the set-up, and Section \ref{sec:formation_thickening} is devoted to the proof of shear-band formation in that model defined on the whole space, and the question of whether these bands thicken. In Section \ref{sec:half-space} we consider the model defined on a half-space, and  show the emergence of shear bands starting at the boundary. In Section \ref{sec:bounded_domains}, we consider shear band formation in bounded domains, adapting the theorems from the preceding sections to that setting. In the appendix \ref{sec:auxiliary_series}, we state and prove several useful but non-standard identities for infinite sums.

\subsection*{Notation}
The positive integers are denoted by $\N=\{1,2,\dots\}$, and the non-negative ones by $\N_0=\{0,1,\dots\}$. The Frobenius inner product between matrices $M,N \in \R^{n \times n}$ is written as $M:N = \Tr(M^TN)$ and the tensor product of two (column) vectors $a,b \in \R^n$ is denoted by $a \otimes b = ab^T \in \R^{n \times n}$. The ball of radius $r>0$ centered at $x \in \R^n$ is written as $B_r(x)$, and the indicator function of a set $A \subset \R^n$ is
\[
\mathbbm{1}_A(x)=\begin{cases}
1 &\text{if $x \in A$},\\
0 &\text{if $x \in \R^n \setminus A$}.
\end{cases}
\]
By $\Hcal^d$ we denote the $d$-dimensional Hausdorff measure. Throughout, we use $C>0$ to denote a generic constant that may change from line to line.

\section{Quasistatic elasto-plasticity with discrete increments}

\label{sec:setting}

In this paper we are interested in the quasistatic evolution of linearly-elastic perfectly-plastic materials with specific restrictions on the plastic increments. We first recall the classical linear elasto-plasticity in the modern framework of rate-independent systems \cite{MiR15}. To this aim, let $\Omega \subset \R^n$ with $n\in\{2,3\}$ be a bounded Lipschitz domain, $u \in \Ucal:=H^1(\Omega;\R^n)$ denote the displacement, and $p \in \Pcal:=L^2(\Omega;\R^{n \times n}_{\rm dev})$ be the plasticity, where $\R^{n \times n}_{\rm dev}$ denotes the symmetric trace-free matrices. The loading and boundary conditions are respectively given by 
\begin{equation*}\label{eq:loadingbc}
f \in W^{1,1}([0,T];L^2(\Omega;\R^n)) \quad \text{and} \quad g \in W^{1,1}([0,T];H^1(\Omega;\R^n))
\end{equation*}
and the Dirichlet conditions are enforced on $\Gamma$, which is a relatively open subset of the boundary $\partial \Omega$ with $\Hcal^{n-1}(\Gamma)>0$. The driving forces behind the quasistatic evolution are the elastic energy $\Ecal:[0,T] \times \Ucal \times \Pcal \to \R \cup \{\infty\}$ given by
  \begin{equation}\label{eq:1}
\Ecal(t,u,p)=\begin{cases}
\displaystyle \int_{\Omega} \frac{1}{2}\C(e(u)-p):(e(u)-p)) -f(t)\cdot u \dd x &\text{if $u=g(t)$ on $\Gamma$},\\
\infty &\text{otherwise},
\end{cases}
\end{equation}
where $e(u)=(\nabla u + \nabla u^T)/2$ is the symmetric gradient and $\C$ the elasticity tensor, and the dissipation potential $\Dcal:\Pcal \to [0,\infty)$ given by
  \begin{equation}\label{eq:2}
\Dcal(q) = \int_{\Omega} |q|_{K} \dd x,
\end{equation}
where $K \subset \R^{n\times n}_{\rm dev}$ is a closed convex even set whose boundary constitutes the yield surface, and
\[
|M|_{K}:= \sup \{M:N \,|\, N \in K\}.
\]
Finally, we also need an initial condition $(u_0,p_0) \in \Ucal \times \Pcal$ such that
\[
\Ecal(0,u_0,p_0) \leq \Ecal(0,\tilde{u},\tilde{p})+\Dcal(\tilde{p}-p_0) \quad \text{for all $(\tilde{u},\tilde{p}) \in \Ucal \times \Pcal$.}
\]
The quasistatic evolution is now characterized as functions $(u,p):[0,T] \to \Ucal \times \Pcal$ such that $(u(0),p(0))=(u_0,p_0)$ and for all $t \in [0,T]$:
\begin{equation}\tag{QE1}
\begin{split}\label{eq:qe}
\Ecal(t,u(t),p(t)) &\leq \Ecal(t,\tilde{u},\tilde{p})+\Dcal(\tilde{p}-p(t)) \quad \text{for all $(\tilde{u},\tilde{p}) \in \Ucal \times \Pcal$,} \\
\Ecal(t,u(t),p(t))&+\mathrm{Diss}_{\Dcal}(p,[0,t]) = \Ecal(0,u(0),p(0))+\int_0^t \partial_s \Ecal(s,u(s),p(s)) \dd s,
\end{split}
\end{equation}
where the dissipation is defined as 
\[
\mathrm{Diss}_{\Dcal}(p,[r,s]):= \sup \left\{ \sum_{j=1}^{N} \Dcal(p(t_j)-p(t_{j-1})) \,\Big|\, r \leq t_0 \leq \ldots \leq t_N \leq s \right\};
\]
Technically speaking, the time-derivative of the energy $\partial_t \Ecal(t,u,p)$ should be interpreted as the derivative at $s=t$ of
\[
s \mapsto \Ecal(s,u-g(t)+g(s),p),
\]
but this plays no role in the restricted formulation below \eqref{eq:qe2}. We mention that the quasistatic evolution phrased in this way may not admit solutions, as the class of displacements has to be relaxed to the space of functions of bounded deformation \cite{DDM06}, but this is not of concern to us here.

Note that for each $p \in \Pcal$, there exists a unique minimizer $u(t,p) \in \Ucal$ such that
\[
\Ecal(t,u(t,p),p) = \min_{u \in \Ucal} \Ecal(t,u,p)<\infty.
\]
Using this, we may reduce the dependence to only the plasticity, by considering the functional $\Jcal:[0,T] \times \Pcal \to \R$ given by
\[
\Jcal(t,p):=\Ecal(t,u(t,p),p).
\]
The quasistatic evolution \eqref{eq:qe} is then equivalent to \emph{global stability} and \emph{energy balance} for all $t \in [0,T]$:
\begin{equation}\tag{QE2}
\begin{split}\label{eq:qe2}
p(t) &\in \argmin_{\Pcal} \Jcal(t,\cdot)+\Dcal(\cdot-p(t)) \\
\Jcal(t,p(t))&+\mathrm{Diss}_{\Dcal}(p,[0,t]) = \Jcal(0,p(0))+\int_0^t \partial_s \Jcal(s,p(s)) \dd s.
\end{split}
\end{equation}

\subsection{Time-incremental minimization problems}
\label{sec:time-incremental}

One of the simplest ways to approximate quasistatic evolutions is to use time-incremental minimization problems. Precisely, we split $(0,T]$ into intervals $(t_{k-1},t_{k}]$ for $k=1,\ldots,N$ with $t_k=kT/N$. Then, for a fixed initial $p_0 \in \Pcal$, we look for $p_k \in  \Pcal$ such that for all $k=1,\ldots,N$
\begin{equation}\tag{IMP}\label{eq:imp}
p_k \in \argmin_{\Pcal} \Jcal(t_k,\cdot)+\Dcal(\cdot-p_{k-1}).
\end{equation}
We then define the left-continuous interpolant $p^{N}:[0,T] \to \Pcal$ as $p^N(0):=p_0$ and
\[
p^{N}(t):=
p_k \quad \text{for $t \in (t_{k-1},t_{k}]$ and $k=1,\ldots,N$}\,,
\]
which, under suitable assumptions, converges as $N \to \infty$ to a solution of the quasistatic evolution \eqref{eq:qe}.

We are now interested in incorporating a restriction on which plastic increments are allowed. We thus introduce a set $\Ical \subset \Pcal$ containing $0$, which represents the admissible increments, and define the new dissipation potential $\Dcal_\Ical:\Pcal \to [0,\infty]$ by
\[
\Dcal_\Ical (q):=\begin{cases}
\Dcal(q)  &\text{if $q \in \Ical$,}\\
\infty &\text{else}.
\end{cases}
\]
For simplicity, we assume that $\Ical$ is weakly closed in $\Pcal$. However, in case the set $\Ical$ is not weakly closed in $\Pcal$, we could still do the same by replacing $\Dcal_{\Ical}$ by its relaxation
\[
\Dcal^{\rm rel}_{\Ical}(q):=\inf \left\{ \liminf_{j \to \infty} \Dcal(q_j) \,\Big|\, (q_j)_j \subset \Ical, \ q_j \rightharpoonup q \ \text{in $\Pcal$}\right\},
\]
to recover lower semicontinuity. In the same way as before, we now define an incremental minimization problem, as a sequence $(p_k)_k \subset \Pcal$ such that for $k=1,\ldots,N$
\begin{equation}\tag{RIMP}\label{eq:rimp}
p_k \in \argmin_{\Pcal} \Jcal(t_k,\cdot)+\Dcal_{\Ical}(\cdot-p_{k-1}).
\end{equation}

Before we analyze \eqref{eq:rimp} mathematically, let us give some explicit examples that have been considered in the literature.
\begin{example}[Shear transformation zones]\label{ex:stz}
Shear transformation zones correspond to plastic shears that are activated in localized zones \cite{Bul94a,Bul94b,Bul94c,KAN18,HoS09}. Since this is used to model amorphous plasticity, there are no preferred shear directions, and the plastic events are interpreted as a collection of unordered atoms sliding along a plane within a localized zone. In the case of spherical zones, the increments correspond to
\[
\Ical=\{0\} \cup \left\{\gamma_0 \mathbbm{1}_{B_{\delta}(y)} \mu \odot \nu \, |\, \mu,\nu \in \S^2, \  \mu \cdot \nu=0, \ y \in C\right\},
\]
where $\gamma_0>0$ is the characteristic shear strain, $\delta>0$ is the radius of the zones, and $C \subset \Omega$ the set of centers of the zones; moreover, the matrix
\[
\mu \odot \nu :=\frac{1}{2}\left(\mu \otimes \nu+\nu \otimes \mu\right)\in \R^{3 \times 3}_{\rm dev},
\]
denotes a shear in the direction $\mu$ along the plane with normal $\nu$.
\end{example}

\begin{example}[Time-discontinuous plasticity]\label{ex:quantized}
In time-discontinuous plasticity, the plastic increments are constrained to have a minimal size $\Delta p_{\rm min}>0$, precisely
\[
\Ical=\{0\} \cup \left\{ q \in \Pcal \,|\, \|q\|_{L^1(\Omega)} \geq \Delta p_{\rm min}\right\}.
\]
The quantization of the increments leads to time-discontinuous evolution with bursts of plasticity at discrete time instances \cite{LKS25}.
\end{example}

\begin{example}[Regularized dislocations]\label{ex:dislocations}
If we consider a subset $\Sigma$ of a slip plane with normal $\nu \in \S^2$, then the plastic strain associated to a slip dislocation on this surface $\Sigma$ can be represented as a measure
\[
q = b \odot \nu \Hcal^{2}\mres_{\Sigma},
\]
with $b \in \R^3$ the Burgers vector. However, it is well-known that the associated displacement fields may form singularities leading to infinite energies in the core radius. Hence, a common strategy is to regularize the dislocation, for example, by convolution
\[
q = \rho_\epsilon * (b \odot \nu \Hcal^{2}\mres_{\Sigma}),
\]
where $(\rho_\epsilon)_\epsilon$ is a standard family of mollifiers. This was considered in \cite{CGO15}, where an asymptotic analysis of the dislocation density is carried out. Inspired by this, we could consider a discrete collection of dislocation slip systems $(b_i,\nu_i,\Sigma_i)$, e.g., those given by the crystalline lattice, and utilize the set of plastic increments
\[
\Ical = \{0\} \cup \bigcup_{i} \left\{\rho_\epsilon * (b_i \odot \nu_i \Hcal^{2}\mres_{\Sigma_i})\right\}.
\]
This is also very closely related to the discrete dislocation approach in \cite{RAO07}, where the optimization is carried out by carrying out the optimal single dislocations one by one.
\end{example}

Since $\Dcal_{\Ical}$ might not satisfy the usual triangle inequality,
\[
\Dcal_{\Ical}(q_1+q_2) \leq \Dcal_{\Ical}(q_1)+\Dcal_{\Ical}(q_2),
\]
it may happen that
\[
p_k \not \in \argmin_{\Pcal} \Jcal(t_k,\cdot)+\Dcal_{\Ical}(\cdot - p_k),
\]
that is, the state $p_k$ does not satisfy the energy stability. Hence, as the time step vanishes, it may occur that multiple transitions happen at a single time $t$ in the limit. To describe this, we assume that there is a $\gamma_0>0$ such that
\begin{equation}
\begin{split}\label{eq:Idiscrete}
 \Dcal(q)\geq \gamma_0 \quad \text{for all $q \in \Ical \setminus \{0\}$} \quad \text{and} \quad \text{$\Ical$ is compact in $\Pcal$.}
\end{split}
\end{equation}
The former condition states that all increments have at least a given size, which corresponds well with the examples mentioned before; the second condition is more technical, and ensures that we do not have to relax the displacements to $BD(\Omega;\R^n)$ (see~\cite{DDM06}) and it is needed for the continuity of the minimizers of the energy-dissipation functional in time. We observe that Example~\ref{ex:stz} satisfies \eqref{eq:Idiscrete} as long as the set of centers $C \subset \Omega$ is compact, and also Example~\ref{ex:dislocations} satisfies \eqref{eq:Idiscrete}. For Example~\ref{ex:quantized}, we can intersect that set of increments with a compact set, e.g., the space of piecewise constant functions on a triangulation of $\Omega$, which is common in finite element simulations \cite{LKS25}. 

We now define the sets $\Acal_k(t,p), \Bcal_k(t,p)$ for $k \in \N_0$ as
\[
\Bcal_0(t,p):=\{p\} \quad \text{and} \quad \Bcal_{k}(t,p):=\bigcup_{\tilde{p} \in \Bcal_{k-1}(t,p)} \argmin_{\Pcal} \Jcal(t,\cdot)+\Dcal_{\Ical}(\cdot - \tilde{p}) \setminus\{\tilde{p}\}.
\]
and
\[
\Acal_0(t,p):=\{p\} \quad \text{and} \quad \Acal_k(t,p):=\Bcal_k(t,p) \cup \bigcup_{\tilde{p} \in \Acal_{k-1}(t,p)} \argmin_{\Pcal} \Jcal(t,\cdot)+\Dcal_{\Ical}(\cdot - \tilde{p}) \cap \{\tilde{p}\}.
\]
The set $\Bcal_k(t,p)$ describes the states that can be reached from $p$ with exactly $k$ increments in $\Ical$, while $\Acal_k(t,p)$ represents the stable states that can be reached in at most $k-1$ steps, together with the (possibly non-stable) states that can be reached in $k$ steps. Note that due to the first part of \eqref{eq:Idiscrete} and the non-negativity of the elastic energy $\Jcal$, there must be some $k_0 \in \N_0$ such that $\Bcal_{k}(t,p)=\emptyset$ for all $k \geq k_0$, and we set 
\[
\Acal(t,p):=\Acal_{k_0}(t,p)=\lim_{k \to \infty} \Acal_k(t,p),
\]
as the collection of all stable states that can be reached from $p$ with increments in $\Ical$.

\medskip

In the upcoming definition, we will use the following notation for a left-continuous piecewise constant function $p:[0,T] \to \Pcal$: For $t\in [0,T)$, 
$p(t^{+}):=\lim_{s \downarrow t} p(s)$ and $J(p)$ denotes the (discrete) jump set of $p$
\[
J(p)=\left\{t \in [0,T)\,|\, p(t) \not = p(t^{+})\right\}.
\]
For $t\in[0,T)$ and $p_-,p_+\in\Pcal$, define
\[
\Tc(t,p_-,p_+):=\inf\Biggl\{
\sum_{i=1}^{m}\Dcal_{\Ical}(p_i-p_{i-1})\,\bigg|\,
\begin{array}{l}
m\in\N,\quad p_0=p_-,\quad p_m=p_+,\\[0.2em]
p_i\in\argmin_{\Pcal} \Jcal(t,\cdot)+\Dcal_{\Ical}(\cdot-p_{i-1}),\\
p_i\neq p_{i-1},\quad i=1,\dots,m
\end{array}\Biggr\}\,,
\]
with the convention that $\Tc(t,p,p)=0$, and $\inf\emptyset=\infty$. Note that $\Tcal(t,p_-,p_+) < \infty$ if and only if $p_+ \in \Bcal_k(t,p_-)$ for some $k \in \N_0$.

We define constrained energetic solution staying close to the concept of energetic solutions of rate-independent systems that we have recalled above.

\begin{definition}[Constrained energetic solution]\label{def:ces}
A left-continuous piecewise constant function $p:[0,T] \to \Pcal$ is said to be a \emph{constrained energetic solution} for the system $(\mathcal P,\mathcal E,\mathcal D,\mathcal I)$ if the following hold for all $t \in [0,T]$:
\begin{align}
p(t) \in \argmin_{\Pcal} \Jcal(t,\cdot)&+\Dcal_{\Ical}(\cdot-p(t));
\tag{S}\label{eq:S}\\
\tag{EI}\label{eq:EI}
 \Jcal(t,p(t))+\sum_{s\in J(p)\cap[0,t]}
 \Tc(s,p(s),p(s^{+}))
&\leq \Jcal(0,p(0))+\int_0^t\partial_s\Jcal(s,p(s))\,\mathrm ds.
\end{align}
\end{definition}

\begin{remark}\label{rem:ces}
a) Equations \eqref{eq:S} and \eqref{eq:EI} are the constrained analogues of the global stability
 and energy balance conditions of \eqref{eq:qe2}, respectively. In contrast to the classical energetic
formulation, the energy inequality \eqref{eq:EI} generally cannot be expected
to be an equality. There are two reasons why the energy inequality need not be an equality, which are best understood by comparing the incremental minimization problem with its continuous-time limit.

First, at a jump time $t$, the limiting evolution records only the initial
and final states $p(t)$ and $p(t^{+})$, whereas the underlying constrained
incremental evolution may pass through a finite sequence of intermediate
states
\[
p(t)=p_0,\ p_1,\ \ldots,\ p_N=p(t^{+}),
\]
where each transition is obtained by one constrained incremental
minimization. Each elementary transition satisfies
\[
\Jcal (t,p_i)+\mathcal D(p_i-p_{i-1})
\le
\Jcal(t,p_{i-1}),
\]
and the inequality may be strict. Consequently, the cumulative energy drop
along the cascade may exceed the cumulative dissipation,
\[
\sum_{i=1}^N
\Bigl(
\Jcal(t,p_{i-1})
-
\Jcal (t,p_i)
-
\mathcal D(p_i-p_{i-1})
\Bigr)
>0\,.
\]

Second, the continuous-time model does not retain the intermediate states. Instead,
it records only the macroscopic jump between its end states and assigns to it
the effective transition cost
\[
\Tc(t,p(t),p(t^{+})),
\]
defined as the minimal dissipation among all admissible minimizing cascades
connecting these states. The energy inequality therefore involves this
effective transition cost rather than the dissipation of any particular
microscopic cascade.

From a physical point of view, the intermediate states represent a rapid
cascade of elementary inelastic events occurring on a time scale that is not
resolved by the quasistatic model. The limiting evolution therefore captures
only the slow evolution between metastable configurations, while the rapid
relaxation during a jump is compressed into an instantaneous transition. The
possible strict inequality in the energy balance may  be interpreted as
reflecting the energetic effect of these unresolved cascades. \medskip

b) We can equivalently replace \eqref{eq:EI} by the condition
\begin{align}\label{eq:oldei}
p(t^{+}) &\in \Acal(t,p(t))
\quad \text{for all $t \in J(p)$.}
\end{align}
Indeed, if \eqref{eq:EI} holds, then we find immediately that 
\[
p(t^{+}) \in \bigcup_{k \in \N} \Bcal_k(t,p(t)) \quad \text{for all $t \in J(p)$},
\]
since otherwise $\Tc(t,p(t),p(t^{+})) = \infty$. Combining this with \eqref{eq:S}, we deduce that \eqref{eq:oldei} holds. For the reverse implication, let $0\leq t_0 < t_1 < \ldots < t_N < T$ be such that $J(p) = \{t_0,\ldots,t_N\}$. Then, for some $t \in (t_i,t_{i+1}]$ we find using the left-continuity of $p$ that
\begin{align}\label{eq:rem}
\begin{split}
\Jcal(t,p(t))=\Jcal(t,p(t_i^{+})) &= \int_{t_i}^{t} \partial_s \Jcal(s,p(t_i^{+}))\dd s + \Jcal(t_i,p(t_i^{+})) \\
&= \int_{t_i}^{t} \partial_s \Jcal(s,p(s))\dd s + \Jcal(t_i,p(t_i^{+})),
\end{split}
\end{align}
with the second equality being the fundamental theorem of calculus. Using \eqref{eq:oldei}, we infer that there are
\[
p(t_i) = p_0 , \, p_1, \, \ldots ,\, p_m= p(t_i^{+})
\]
with $p_{j} \in \argmin_{\Pcal} \Jcal(t_i,\cdot) +\Dcal_{\Ical}(\cdot - p_{j-1})$ for $j=1,\ldots,m$. Hence, the definition of $\Tc$ yields
\begin{align*}
\Tc(t_i,p(t_i),p(t_i^{+})) &\leq \sum_{j=1}^{m} \Dcal_{\Ical}(p_j-p_{j-1}) 
\leq \sum_{j=1}^{m} \Jcal(t_i,p_{j-1})-\Jcal(t_i,p_j)\\
& = \Jcal(t_i,p(t_i))-\Jcal(t_i,p(t_i^{+})).
\end{align*}
Adding this to \eqref{eq:rem} results in
\[
\Jcal(t,p(t)) + \Tc(t_i,p(t_i),p(t_i^{+})) \leq \Jcal(t_i,p(t_i)) +\int_{t_i}^{t} \partial_s \Jcal(s,p(s))\dd s.
\]
Continuing like this, we arrive at \eqref{eq:EI} as desired.
\end{remark}

 We now prove that the solution of \eqref{eq:rimp} indeed converges to a constrained energetic solution of $(\mathcal P,\mathcal E,\mathcal D,\mathcal I)$. 
\begin{theorem}
  \label{thm:existence_rimp} Let $\Omega\subset\R^n$ be a bounded Lipschitz domain, $\mathcal P=L^2(\Omega;\R^{n\times n}_\mathrm{dev})$, $\mathcal E$ as in \eqref{eq:1} with loading $f \equiv 0$, $\mathcal I\subset\mathcal P$ a set containing $\{0\}$,  $\mathcal D$ as in \eqref{eq:2}, and 
assume that \eqref{eq:Idiscrete} is satisfied. Then, for any initial plasticity $p_0 \in \Pcal$ satisfying
 \begin{equation}\label{eq:4}
p_0\in \argmin_{\Pcal}  \Jcal (0,\cdot)+\mathcal D_{\mathcal I}(\cdot-p_0)\,,
\end{equation}
there exist solutions to \eqref{eq:rimp}. Moreover, the left-continuous interpolants $p^{N}$ of any solutions to \eqref{eq:rimp} converge, up to a non-relabeled subsequence, weak* in $BV([0,T];\Pcal)$ to a constrained energetic solution $p$ of $(\mathcal P,\mathcal E,\mathcal D,\mathcal I)$.
\end{theorem}

We precede the proof of Theorem  \ref{thm:existence_rimp} by a remark concerning our choice of the technical setting.

\begin{remark}
For clarity, we have restricted ourselves to the case of vanishing body forces $f\equiv 0$, thereby isolating the effect of the constrained set of admissible inelastic increments. We expect that the analysis extends to non-zero loading under a suitable \emph{safe-load condition}, as is standard in the theory of perfect plasticity, see \cite{DDM06,MiR15}. 
Since these arguments would largely follow the well-established techniques developed in the cited references, while being largely orthogonal to the novel aspect of the present work—the introduction of constraints on admissible inelastic increments—we have chosen to work in the present simplified setting. Another possibility would be to introduce hardening into the model. The resulting coercivity of the energy with respect to the inelastic variable provides a direct route to the required a priori estimates, avoiding the need for a safe-load condition.
\end{remark}

\begin{proof}[Proof of Theorem \ref{thm:existence_rimp}]
The existence of solutions to \eqref{eq:rimp} follows from the direct method using the weak lower semicontinuity of $\Jcal(t,\cdot)$ and $\Dcal$ together with the $L^2$-coercivity coming from the second part of \eqref{eq:Idiscrete} and the non-negativity of $\Jcal$ given that $f \equiv 0$. We denote the time steps by $t^N_k:=kT/N$, the discrete solutions by $(p^N_k)_k$, where $p^N_0=p_0$ for all $N$, and the left-continuous interpolations are denoted by $p^N$. We now split the proof into two steps. \medskip

\textit{Step 1: Compactness.} We first prove suitable energy estimates that show that only finitely many transitions can happen. To do this, we compute for a.e.~$t \in (0,T)$
\begin{align*}
\partial_t \Jcal(t,p) = \int_{\Omega} \C(e(u(t,p))-p): e(\partial_t u(t,p)) \dd x = \int_{\Omega} \C(e(u(t,p))-p):e(g'(t))\dd x,
\end{align*}
where the second equality uses the Euler-Lagrange equations applied to the test function $\partial_t u(t,p)-g'(t)$, which is zero on $\Gamma$. Hence, we can estimate
\begin{align*}
\abs{\partial_t \Jcal(t,p)} &\leq C\norm{g'(t)}_{H^1(\Omega;\R^n)}\norm{e(u(t,p))-p}_{L^2(\Omega;\R^{n \times n})} \\
&\leq C\norm{g'(t)}_{H^1(\Omega;\R^n)} \left(\norm{e(u(t,p))-p}_{L^2(\Omega;\R^{n \times n})}^2+1\right) \\
&\leq \lambda(t) \left(\Jcal(t,p)+1\right)
\end{align*}
for some suitable $\lambda \in L^1([0,T])$. If we now define $\Jcal_k:=\Jcal(t_k,p_k)$ and $\Dcal_k:=\Dcal_{\Ical}(p_k-p_{k-1})$, then we get that
\[
\Jcal_k-\Jcal_{k-1} + \Dcal_k \leq \int_{t_{k-1}}^{t_k} \partial_s \Jcal(s,p_{k-1})\dd s \leq \int_{t_{k-1}}^{t_k} \lambda(s)\left(\Jcal(s,p_{k-1})+1\right)\dd s.
\] 
Following the estimates using Gr\"onwall's inequality as in \cite[page 53]{MiR15}, we arrive at
\[
1+\Jcal_N+\sum_{k=1}^{N} \Dcal_k \leq (1+\Jcal_0)e^{\Lambda(T)} =:K <\infty \quad
\text{with} \quad \Lambda(s):=\int_0^s \lambda(s) \dd s.
\]
Since \eqref{eq:Idiscrete} yields that $\Dcal_k$ must either be zero or larger or equal than $\gamma_0$, we find that there is a $M \in \N$ independent of $N$ such that 
\[
J_N:=\left\{ t \in [0,T) \,|\, p^N(t) \not = p^N(t^{+}) \right\}  \quad \text{satisfies} \quad \#J_N \leq M.
\]
By the boundedness of $\Ical$ from the second part of \eqref{eq:Idiscrete}, we obtain
\[
\sum_{k=1}^N \|p_k^N-p_{k-1}^N\|_{L^2(\Omega)}
\leq
C\sum_{k=1}^N \Dcal(p_k^N-p_{k-1}^N)
\leq C,
\]
where the constant is independent of $N$. Hence, the left-continuous interpolants
$p^N$ are uniformly bounded in
$BV([0,T];\Pcal)$. By Helly's selection principle for Banach-valued
$BV$ functions, after passing to a subsequence, we may assume that
$p^N(t)\rightharpoonup p(t)\in\mathcal P$ weakly for every
$t\in[0,T]$; in fact, by redefining $p$ at a finite number of $t \in (0,T]$, we may assume that $p$ is left-continuous. Moreover, by possibly choosing further subsequences, we may assume that $\#J_N = \ell$ for all $N$, that is, $J_N$ consists of
\[
0 \leq s^N_1 < \ldots < s^N_\ell < T,
\]
and we have that
\begin{align}
\begin{split}\label{eq:compactness}
s^N_i \to s_i \quad&\text{for all $i=1,\ldots,\ell$,}\\
q^N_i:=p^N((s^N_i)^{+})-p^N(s^N_{i}) \to q_i \quad &\text{for all $i=1,\ldots,\ell$,}
\end{split}
\end{align}
where the second part uses the compactness of $\Ical$ from \eqref{eq:Idiscrete}. This implies in particular that 
\[
J(p)\subset S:=\{s_1 , \ldots , s_\ell\}\,,
\]
where the $s_i$ need not be distinct, and a limiting time $s < T$ gives rise to a
jump only if the sum of the weak limits of all increments whose transition
times converge to $s$ is nonzero. Furthermore, 
\begin{equation}\label{eq:compactness2}
p^N(t) \to p(t) \quad \text{for all $t \in [0,T] \setminus S$.}
\end{equation}
 \smallskip

\textit{Step 2: $p$ solves \eqref{eq:S} and \eqref{eq:EI}.} 
We start with proving the global stability. To this aim, take $t \in (0,T) \setminus S$ and consider the sequence $(k_N)_N$ such that $t^N_{k_N-1} < t \leq t^N_{k_N}$. Then, because there is no jump at $t^N_{k_N}$ for all large $N$, we get by definition of \eqref{eq:rimp} that
\[
p^N(t) \in \argmin_{\Pcal} \Jcal(t^{N}_{k_N},\cdot) + \Dcal_{\Ical}(\cdot - p^N(t)).
\]
In particular, it holds for all $\tilde{p} \in \Pcal$ that
\[
0 \leq \Jcal(t^{N}_{k_N},\tilde{p})-\Jcal(t^{N}_{k_N},p^N(t)) + \Dcal_{\Ical}(\tilde{p} - p^N(t)).
\]
Using \eqref{eq:compactness2} and the continuity of the boundary condition in time, we find by taking the limit
\[
0\leq \Jcal(t,\tilde{p})-\Jcal(t,p(t)) + \Dcal_{\Ical}(\tilde{p} - p(t)) \quad \text{for all $\tilde{p} \in \Pcal$.}
\]
This proves global stability  for all $t \in (0,T) \setminus S$. The stability in $t=0$ follows from \eqref{eq:4}. For the case $t \in S$, we simply use that $p(t)$ is left-continuous and piecewise constant together with the continuity of the energy in time.

To characterize the jumps, we consider $t \in J(p)$ such that $s_{i-1}<t=s_i=\ldots=s_{i+m-1} < s_{i+m}$ for some $i \in \{1,\ldots,\ell\}$. We may delete from that sequence all $s_j$ that correspond to a jump of height zero via $q_j^N\to 0$. Then, with $q_i$ from \eqref{eq:compactness}, we show that 
\begin{equation}\label{eq:induction}
r_k:=p(t)+\sum_{j=0}^{k-1} q_{i+j} \in \Bcal_k(t,p(t)) \quad \text{for all $k=0,\ldots,m$},
\end{equation}
via induction on $k$. The case $k=0$ is by definition, so suppose now that $r_k \in \Bcal_k(t,p(t))$ for $k \in \{0,\ldots,m-1\}$. We have to show that
\begin{equation}\label{eq:jumpstability}
r_k+q_{i+k} \in \argmin_{\Pcal} \Jcal(t,\cdot)+\Dcal_{\Ical}(\cdot - r_k).
\end{equation}
With $r_k^{N}:=p^N(s_i^{N})+\sum_{j=0}^{k-1} q_{i+j}^N$, we find that
\[
r_k^N+q_{i+k}^N \in \argmin_{\Pcal} \Jcal(s^N_{i+k},\cdot)+\Dcal_{\Ical}(\cdot - r^N_k),
\]
and thus, for every $\tilde{p} \in \Pcal$ we find
\[
\Jcal(s^N_{i+k},r_k^N+q_{i+k}^N)+\Dcal_{\Ical}(q_{i+k}^N)\leq \Jcal(s^N_{i+k},\tilde{p})+\Dcal_{\Ical}(\tilde{p} - r^N_k).
\]
Using \eqref{eq:compactness} and \eqref{eq:compactness2}, we may take the limit $N \to \infty$ to find that
\[
\Jcal(t,r_k+q_{i+k})+\Dcal_{\Ical}(q_{i+k})\leq \Jcal(t,\tilde{p})+\Dcal_{\Ical}(\tilde{p} - r_k) \quad \text{for all $\tilde{p} \in \Pcal$,}
\]
which establishes \eqref{eq:jumpstability}. This proves \eqref{eq:induction}, and, in particular, shows that
\[
p(t^{+}) = r_m \in \Bcal_m(t,p(t)).
\]
Since $p(t^{+})$ also satisfies global stability, we must have $p(t^{+}) \in \Acal(t,p(t))$ as well, which finishes the proof.
\end{proof}
\begin{remark}
It is possible that not every constrained energetic solution  can be recovered from a sequence of solutions of \eqref{eq:rimp}; for example, it may happen that $\Acal(t,p_0)\setminus\{p_0\}$ is only non-empty at an irrational $t \in [0,T]$, which will never coincide with $t_k$ for some $k$. Hence, the time-incremental solution is only constantly $p_0$, while a constrained energetic solution can have transitions. On the other hand, if we allow the partition of $[0,T]$ to be arbitrary (including allowing certain time steps to be zero), then we can trivially recover any constrained energetic solution  with solutions of \eqref{eq:rimp}. In fact, we can match the solution exactly for $N$ large enough.
\end{remark}

Since the concept of constrained energetic solution  is completely tied to the sets $\Acal(t,p)$, cf.~Remark~\ref{rem:ces}~b), it is interesting to understand this set better. It is clear that for $p \in \Pcal$ the set
\[
\argmin_{\Pcal} \Jcal(t,\cdot) +\Dcal_{\Ical}(\cdot - p),
\]
plays an important role. Introducing the functional 
\begin{equation}\label{eq:changefunc}
\Fcal_p(t,q):=\Jcal(t,p+q)-\Jcal(t,p) +\Dcal_{\Ical}(q),
\end{equation}
we find
\[
\argmin_{\Pcal} \Jcal(t,\cdot) +\Dcal_{\Ical}(\cdot - p) = \{p\}+\argmin_{q \in \Pcal} \Fcal_p(t,q).
\]
The reason to do this is the following representation of $\Fcal_p$, showing that the change in energy is only dependent on the increment $q$, and the current stress integrated against the increment $q$.

\begin{lemma}[Energy change]
For $p,q \in \Pcal$, it holds that
\[
\Fcal_p(t,q) = -\int_{\Omega}\sigma(t,p):q\dd x + \Jcal_0(q)+\Dcal_{\Ical}(q),
\]
where $\sigma(t,p):=\C (e(u(t,p))-p)$ is the stress and $\Jcal_0(q)$ is the reduced elastic energy with zero loading and boundary conditions, i.e.,
\begin{equation}\label{eq:J0}
\Jcal_0(q):=\min\left\{ \int_{\Omega} \frac{1}{2}\C (e(w)-q):(e(w)-q)\dd x \, \Big|\, w \in H^1(\Omega;\R^n), \ w=0 \ \text{on $\Gamma$}\right\}.
\end{equation}
\end{lemma}
\begin{proof}
Let $w(q) \in H^1(\Omega;\R^n)$ be the minimizer of \eqref{eq:J0}, then it follows from the linearity of the Euler-Lagrange equations that $u(t,p+q) = u(t,p)+w(q)$. Hence, we find after some algebraic manipulations
\begin{align*}
\Jcal(t,p+q)-\Jcal(t,p) &= \int_{\Omega} \C(e(u(t,p))-p):(e(w(q))-q) -f(t)\cdot w(q) \\
&\qquad\qquad+\frac{1}{2}\C(e(w(q))-q):(e(w(q))-q)\dd x \\
&=\int_{\Omega} -\C(e(u(t,p))-p):q+\frac{1}{2}\C(e(w(q))-q):(e(w(q))-q)\dd x \\
&=-\int_{\Omega} \sigma(t,p):q \dd x + \Jcal_0(q),
\end{align*}
where the second line uses the weak Euler-Lagrange equation of $u(t,p)$ with test function $w(q)$.
\end{proof}

\begin{remark}
In classical rate-independent plasticity, incremental stability is equivalent to requiring that the thermodynamic force associated with the inelastic variable belong to the elastic domain (or yield set). This condition is characterized by the dissipation potential through the inclusion $-\partial_p\mathcal E(u,p)\in\partial \mathcal D(0)$, and is given by 
\[
K=\left\{\sigma \,\Big|\, \int_\Omega\sigma:q\d x\leq\mathcal D(q) \quad\forall q\in\mathcal P\right\}\,.
\]

In the present framework, stability is tested only with respect to the prescribed family of admissible increments. Consequently, the classical elastic domain is replaced by an effective constrained stability domain, consisting of those stress states for which no admissible increment decreases the total energy. The effective constrained stability domain is given by 
\[
K_{\mathcal I}=\left\{\sigma \,\Big|\, \int_\Omega\sigma:q\d x\leq\mathcal D(q)+\Jcal_0(q) \quad\forall q\in\mathcal I\right\}\,. \qedhere
\]
\end{remark}

\section{Set-up with localized plastic increments}

\label{sec:localized_plastic}

The rest of the paper focuses on understanding the set $\Acal(t,p)$ in certain prototypical situations where the set $\Ical$ is given by localized plastic events as in Example~\ref{ex:stz}. Since this is a completely static problem, we may ignore the time dependence, and consider zero loading and a fixed boundary condition.

To facilitate exact computations, we consider the case of anti-plane shear, also called mode-III shear. In this setting, the displacement $u$ is only non-zero in the third component and is independent of the last variable, that is,
\[
u(x) = (0,0,u_3(x_1,x_2)), \quad \text{for $x \in \Omega$}.
\]
By identifying all the quantities with their dimensionally reduced versions, we can reformulate the problem as a scalar two-dimensional problem on a bounded Lipschitz domain $\Omega \subset \R^2$, $\Ucal=H^1(\Omega)$, $\Pcal = L^2(\Omega;\R^2)$, elastic energy given by
\[
\Ecal:\Ucal \times \Pcal \to [0,\infty], \quad \Ecal(u,p):= \begin{cases}
\displaystyle \int_{\Omega} \frac{\mu}{2}\abs{\nabla u -p}^2\dd x &\text{if $u=g$ on $\Gamma$,}\\
\infty &\text{otherwise},
\end{cases}
\]
with $\mu>0$ the bulk modulus, and dissipation potential
\[
\Dcal(q):\Pcal \to [0,\infty), \quad \Dcal(q)=\sigma_c\int_{\Omega}  |q| \dd x,
\]
with $\sigma_c >0$ the yield strength. 

For the admissible increments, cf.~Example~\ref{ex:stz}, we fix $\gamma_0,\delta>0$ and consider the set
\begin{equation}\label{eq:increments}
\Ical=\Ical_{\gamma_0,\delta}:=\{0\}\cup\left\{ \gamma_0 \mathbbm{1}_{B_\delta(y)} \nu \,|\, \nu \in \S^1, \ y \in \delta \Lambda \cap \Omega_\delta\right\} \quad \text{with} \quad \Omega_\delta:=\{x \in \Omega\,:\, \dist(x,\partial\Omega)\geq \delta\},
\end{equation}
for some suitable lattice $\Lambda \subset \R^2$; we will consider square and triangular lattices. For such increments $q_{\nu,y}:=\gamma_0 \mathbbm{1}_{B_\delta(y)} \nu$, we can compute that the functional in \eqref{eq:changefunc} leads to
\[
\Fcal_p(q_{\nu,y})=-\gamma_0\int_{B_\delta(y)} (\sigma(p) \cdot \nu -\sigma_c)\dd x + \gamma_0^2 \Jcal_0(\mathbbm{1}_{B_\delta(y)}\nu).
\]
In fact, in the case that $p$ and $q_{\nu,y}$ have disjoint support, we may use the fact that the stress $\sigma(p)$ is harmonic outside the support of $p$, to infer with the mean value property
\begin{equation}\label{eq:changestz}
\Fcal_p(q_{\nu,y})=-\gamma_0|B_\delta(0)| (\sigma(p)(y) \cdot \nu -\sigma_c) + \gamma_0^2 \Jcal_0(\mathbbm{1}_{B_\delta(y)}\nu).
\end{equation}
With this formulation, the sets $\Acal_k(p)$ and $\Bcal_k(p)$ now read
  \begin{equation}\label{eq:12}
\Bcal_0(p):=\{p\} \quad \text{and} \quad \Bcal_{k}(p):=\bigcup_{\tilde{p} \in \Bcal_{k-1}(p)} \left(\{\tilde{p}\}+\argmin_{\Pcal} \Fcal_{\tilde{p}}\right) \setminus\{\tilde{p}\}
\end{equation}
and
\begin{equation}\label{eq:Acal}
\Acal_0(p):=\{p\} \quad \text{and} \quad \Acal_k(p):=\Bcal_k(p) \cup \bigcup_{\tilde{p} \in \Acal_{k-1}(p)} \left(\{\tilde{p}\} + \argmin_{\Pcal} \Fcal_{\tilde{p}}\right) \cap \{\tilde{p}\}.
\end{equation}

In order to determine these sets, we need the updates to the displacements and stresses resulting from a plastic increment. If we define $w_{\nu,y} \in H^1(\Omega)$ as the unique weak solution of
\begin{equation}\label{eq:displacementupdate}
\begin{cases}
\Div (\nabla w_{\nu,y}-\mathbbm{1}_{B_\delta(y)}\nu)=0 &\text{on $\Omega$},\\
w_{\nu,y}=0 &\text{on $\Gamma$,}\\
(\nabla w_{\nu,y}-\mathbbm{1}_{B_\delta(y)}\nu) \cdot n_{\Omega}=0 &\text{on $\partial \Omega \setminus \Gamma$,}
\end{cases}
\end{equation}
with $n_{\Omega}$ the outward unit normal to $\partial \Omega$, we infer from the superposition principle that
\begin{equation}\label{eq:newdisplacementstress}
u(p+q_{\nu,y}) = u(p) + \gamma_0 w_{\nu,y} \quad \text{and} \quad \sigma(p+q_{\nu,y}) = \sigma(p) + \gamma_0\mu(\nabla w_{\nu,y}-\mathbbm{1}_{B_\delta(y)}\nu).
\end{equation}
In addition, we also find that
\begin{equation}\label{eq:quadraticcorr}
\begin{split}
\Jcal_0(\mathbbm{1}_{B_\delta(y)}\nu) &= \int_{\Omega} \frac{\mu}{2}|\nabla w_{\nu,y}-\mathbbm{1}_{B_\delta(y)}\nu|^2 \dd x\\
&=\frac{-\mu}{2}\int_{B_\delta(y)} \nu \cdot (\nabla w_{\nu,y}(x)-\nu)\dd x \\
&=\frac{\mu}{2}|B_\delta(0)| \left(1-\nu \cdot \nabla w_{\nu,y}(y)\right),
\end{split}
\end{equation}
where we used the Euler-Lagrange equation \eqref{eq:displacementupdate} with test function $w_{\nu,y}$ in the second line, and the mean value property in the third. This formula makes \eqref{eq:changestz} fully explicit in terms of the current stress and the displacement increments $w_{\nu,y}$.

Since $w_{\nu,y}$ plays a key role, we are also interested in cases where $\Omega$ is unbounded and explicit formulas are available, such as a full or half space, with boundary conditions at infinity. For such situations the elastic energy may be infinite, but the change in energy functional $\Fcal_p$ and the displacement update $w_{\nu,y}$ are still well-defined. Hence, the sets $\Acal_{k}(p)$ are still meaningful. In addition, these settings will help us understand the bounded domain case via perturbation methods.

\section{Formation and thickening of shear bands in full space}

\label{sec:formation_thickening}
In this section we consider the case that $\Omega=\R^2$  with linear boundary conditions at infinity. Of course, in this case, the energy may become ill-defined.
Therefore we directly work with the energy-change functional \eqref{eq:changestz}, which is still well-defined in this context, and the definition of the permissible sets 
$\Acal_k(p), \Bcal_k(p)$ for $k \in \N_0$ is the one given in \eqref{eq:12}, \eqref{eq:Acal}.

We show that the admissible plastic increments form shear bands with unbounded length. In fact, depending on the lattice, we show that after the formation of shear bands either a new band forms next to it, or the activation of plastic zones halts.

Precisely, we consider the linear boundary condition
\[
g(x) = \gamma x_2
\]
at infinity, that is, we work on the space
\[
\mathcal U=\left\{u\in H^1_{\mathrm{loc}}(\R^2):\lim_{x_2\to\infty} \|u(\cdot,x_2)-\gamma x_2\|_{L^2(\R)}= 0\right\}\equiv \{ x\mapsto \gamma x_2\}+\mathcal U_0\,.
\]
The lattices we are interested in are given by 
\[
\Lambda_{s}:=\{  2(m,n) \,:\, m,n \in \Z\}=(2\Z)^2,
\]
and
\[
\Lambda_{t}:=\{((2m,0)+(n,\sqrt{3}n)) \,:\, m,n \in \Z\}.
\]
It can be verified that the unique solution of $\Div (\nabla w_{\nu,y}-\mathbbm{1}_{B_\delta(y)}\nu)=0$ in $\mathcal U_0$  is given by
\begin{equation}\label{eq:displcamentr2}
w_{\nu,y}(x)=v_\nu(x-y) \quad \text{with} \quad v_\nu(x):=\begin{cases}
\dfrac{1}{2}\nu \cdot x &\text{for $x \in B_\delta(0)$},\\[2ex]
\dfrac{\delta^2}{2} \dfrac{\nu \cdot x}{|x|^2} &\text{for $x \in B_\delta(0)^c$},
\end{cases}
\end{equation}
which yields from \eqref{eq:quadraticcorr} that
\begin{equation}\label{eq:fpr2}
\Fcal_p(q_{\nu,y})=-\gamma_0|B_\delta(0)|\left( \sigma(p)(y) \cdot \nu -\sigma_c - \frac{\gamma_0\mu}{4} \right).
\end{equation}
It is clear that for a fixed $y$, this quantity is minimized when $\nu = \sigma(p)(y)/\abs{\sigma(p)(y)}$, leading to the characterization
\begin{equation}\label{eq:minimizerschar}
q_{\nu,y} \in \argmin_{\Pcal} \Fcal_p \quad \Leftrightarrow \quad \nu = \frac{\sigma(p)(y)}{\abs{\sigma(p)(y)}} \quad \text{and} \quad |\sigma(p)(y)|=\max_{\tilde{y} \in \delta \Lambda} |\sigma(p)(\tilde{y})| \geq \sigma_c +\frac{\gamma_0\mu}{4}.
\end{equation}

\subsection{Formation of shear bands}

With these preparations we prove the following result.

\begin{theorem}\label{th:shearband}
Let $\Omega =\R^2$, and the boundary condition at infinity given by $g(x)=\gamma x_2$ with
\[
\gamma = \frac{\sigma_{c}}{\mu}+\frac{\gamma_0}{4} \geq \gamma_0.
\]
Suppose the admissible increments $\Ical$ are given by \eqref{eq:increments} with $\Lambda \in \{\Lambda_s,\Lambda_t\}$.  Then, it holds that
\[
\Acal_{k}(0) = \{0\} \cup \left\{ \gamma_0\sum_{j=0}^{k-1} e_2\mathbbm{1}_{B_\delta(x_0+2j\delta e_1)} \ \Big| \ \text{$x_0 \in \delta \Lambda$}\right\} \quad \text{for any $k \in \N$},
\]
with $\Acal_k(p)$ as in \eqref{eq:Acal}.
\end{theorem}
\begin{proof}
By rescaling we may assume without loss of generality that $\delta=1$. Note that for initial plasticity $p_0=0$, we have that the displacement is given by $u(p_0) = g \in H^1(\R^2)$. In particular, the stress $\sigma(p_0) \equiv \mu \gamma e_2$ is uniform and satisfies
\[
|\sigma(p_0)|= \mu \gamma = \sigma_c + \frac{\mu \gamma_0}{4}.
\]
Hence, \eqref{eq:minimizerschar} is satisfied for any $y \in \Lambda$ and we deduce
\[
\Acal_1(0) = \{0\} \cup \left\{ \gamma_0e_2\mathbbm{1}_{B_1(x_0)} \,|\, x_0 \in \Lambda\right\}.
\]
We now proceed by induction and assume that
\[
\Acal_{k}(0)=\{0\} \cup \left\{ \gamma_0\sum_{j=0}^{k-1} e_2\mathbbm{1}_{B_1(x_0+2j e_1)} \ \Big| \ \text{$x_0 \in \Lambda$}\right\}
\]
for some $k \geq 1$. To determine $\Acal_{k+1}(0)$, we take $p \in \Acal_{k}(0) \setminus \{0\}$ and need to determine $\argmin \Fcal_p$, which by \eqref{eq:minimizerschar} reduces to finding where the modulus of the stress $\sigma(p)$ is maximized. By translation invariance, we may assume that
\[
p=\gamma_0\sum_{j=0}^{k-1} e_2\mathbbm{1}_{B_1(2j e_1)},
\]
which yields by \eqref{eq:newdisplacementstress} and \eqref{eq:displcamentr2}
\begin{equation}\label{eq:stressr2}
\frac{\sigma(p)(x)}{\mu} = \gamma e_2 + \gamma_0 \sum_{j=0}^{k-1} \nabla v(x-2j e_1)-e_2\mathbbm{1}_{B_1(2j e_1)}(x),
\end{equation}
where
\begin{equation}\label{eq:eshelbystress}
\nabla v(x) = \begin{cases}
\left(0,\dfrac{1}{2}\right) &\text{if $x \in B_1(0)$},\\[2ex]
\dfrac{1}{2}\left(\dfrac{-2x_1x_2}{\abs{x}^4},\dfrac{x_1^2-x_2^2}{\abs{x}^4}\right) &\text{if $x \in \R^2 \setminus B_1(0)$}.
\end{cases}
\end{equation}
Equivalently, with $z:=x_1-x_2 i \in \C$ we find that 
\[
\partial_1 v(x)i + \partial_2v(x) = \frac{1}{2}\frac{z^2}{\abs{z}^4} \quad \text{for $x \in \R^2 \setminus B_1(0)$}.
\]
Hence, $\nabla v(x)$ is a vector of length $1/(2\abs{x}^2)$, which rotates depending on the angle with the origin; it points in the $-e_2$ direction when $x_1=0$ and the $e_2$ direction when $x_2=0$. We aim to show that $|\sigma(p)|$ is maximized precisely at $-2 e_1$ and $2k e_1$, which would yields the desired result. We distinguish between the two lattices. \medskip

\textit{Case 1: $\Lambda = \Lambda_s$.} We can compute that
  \begin{equation}
\label{eq:3}
\frac{|\sigma(p)(-2e_1)|}{\mu} = \frac{|\sigma(p)(2ke_1)|}{\mu} = \gamma + \frac{\gamma_0}{8} \sum_{j=1}^{k} \frac{1}{j^2} \geq \gamma+\frac{\gamma_0}{8}.
\end{equation}
To compare this with other points, consider first $x=(2m,2n) \in \Lambda_s$ with $m \not \in \{0,\ldots,k-1\}$, which is outside any activated plastic zone. Then, we infer that
\begin{align*}
|\sigma(p)(x)| &= \mu \left| \gamma e_2 + \gamma_0 \sum_{j=0}^{k-1} \nabla v(x-2j e_1) \right| \\
&\leq |\sigma(p)(2m,0)| \leq |\sigma(p)(-2e_1)|,
\end{align*}
where the first inequality uses that the functions $\nabla v(\cdot - 2je_1)$ each increase in length and point in the preferred direction $e_2$ when moving from $(2m,2n)$ to $(2m,0)$. Similarly for the second inequality, these functions become larger in length. Also note that for $x=(2m,0)$ with $m\not \in\{-1,k\}$, we have
\[
\frac{|\sigma(p)(-2e_1)|-|\sigma(p)(x)|}{\mu}\geq \frac{|\sigma(p)(-2e_1)|-|\sigma(p)(-4e_1)|}{\mu}\geq c\gamma_0
\]
for some positive constant $c$.

\medskip

Secondly, we consider $x=(2m,2n) \in \Lambda_s$ with $m \in \{0,\ldots,k-1\}$ and $n\not =0$. Again, we are outside any activated zone, so we find
\[
\frac{\sigma(p)(x)}{\mu}= (0,\gamma) + \frac{\gamma_0}{8} \left(\sum_{j=0}^{k-1} \frac{-2(m-j)n}{((m-j)^2+n^2)^2},\sum_{j=0}^{k-1}\frac{(m-j)^2-n^2}{((m-j)^2+n^2)^2}\right).
\]
Estimating the sum in the first component, we have that
  \begin{equation}\label{eq:comp1estimate}
\abs*{\sum_{j=0}^{k-1} \frac{-2(m-j)n}{((m-j)^2+n^2)^2}} \leq \sum_{j=1}^{\infty} \frac{2j|n|}{(j^2+n^2)^2}
\leq \sum_{j=1}^\infty\max_{t \geq 0}\frac{2jt}{(j^2+t^2)^2}
 = \frac{3\sqrt{3}}{8}\sum_{j=1}^{\infty} \frac{1}{j^2} = \frac{\sqrt{3}\pi^2}{16}\,.
\end{equation}

For the sum in the second component, we first note the lower bound
  \begin{equation}\label{eq:lowercomp2estimate}
  \begin{split}
\sum_{j=0}^{k-1} \frac{(m-j)^2-n^2}{((m-j)^2+n^2)^2}&\geq \sum_{j=-n}^{n} \frac{j^2-n^2}{(j^2+n^2)^2}\\
&\geq -\frac{1}{n^2}-2 \sum_{j=1}^{n} \min_{t \geq 0} \frac{j^2-t^2}{(j^2+t^2)^2}\\
&= -\frac{1}{n^2}-\frac{1}{4} \sum_{j=1}^{n} \frac{1}{j^2} \\
&\geq \min\{-1,-1/4-\pi^2/24\}=-1\,.
\end{split}
\end{equation}
The upper bound inequality 
  \begin{equation}
\sum_{j=0}^{k-1} \frac{(m-j)^2-n^2}{((m-j)^2+n^2)^2}\leq 0\label{eq:uppercomp2estimate}
\end{equation}
follows from  Lemma \ref{lem:sumleq0}, and hence we obtain
\[
\sum_{j=0}^{k-1} \frac{(m-j)^2-n^2}{((m-j)^2+n^2)^2}\in [-1,0]\,.
\]
Now we set $a_s:=\frac{\sqrt{3}\pi^2}{128}$, $t=\frac{\gamma_0}{\gamma}\leq 1$, and note that the above bounds imply
\[
\frac{|\sigma(p)(x)|}{\mu}\leq \gamma \sqrt{1+a_s^2t^2}\,.
\]
Combining this with \eqref{eq:3}, we obtain 
\[
\frac{|\sigma(p)(x)|-|\sigma(p)(-2 e_1)|}{\mu \gamma_0}\leq \frac{\sqrt{1+a_s^2 t^2}-1-\frac18t}{t}\,.
\]
 The function on the right extends continuously to $t=0$ with value
 $-1/8$ and is strictly negative  on $[0,1]$ (since the numerator is convex, and negative in $\{0,1\}$). Therefore
we have shown that for all $x=(2m,2n)$ with $m\in\{0,\dots,k-1\}$, $n\neq 0$,
\[
\frac{|\sigma(p)(-2 e_1)|-|\sigma(p)(x)|}{\mu }>c\gamma_0 \,,
\]
if $c$ is chosen small enough. In particular, $c$ may be chosen independently of $k$. 

Finally, we consider the case $x=(2m,0) \in \Lambda_s$ with $m \in \{0,\ldots,k-1\}$, i.e., $x$ is the center of an already activated zone. Then, we compute that
\[
\frac{|\sigma(p)(x)|}{\mu} = \gamma-\frac{\gamma_0}{2} + \frac{\gamma_0}{8}\sum_{j\in\{0,\ldots,k-1\}\setminus\{m\}} \frac{1}{(m-j)^2}.
\]
Hence, we find
\[
\frac{|\sigma(p)(-2e_1)|-|\sigma(p)(x)|}{\mu} = \frac{\gamma_0}{2} + \frac{\gamma_0}{8}\left(\sum_{j=1}^{k} \frac{1}{j^2}-\sum_{j\in\{0,\ldots,k-1\}\setminus\{m\}} \frac{1}{(m-j)^2}\right) \geq \frac{\gamma_0}{2}-\frac{\gamma_0}{8}\frac{\pi^2}{6}>c\gamma_0\,,
\]
if $c$ is chosen small enough.

All in all, we conclude that
\[
\argmax_{x \in \Lambda} |\sigma(p)(x)| = \{-2e_1,2ke_1\}.
\]
Since also $|\sigma(p)(-2e_1)|>\mu \gamma = \sigma_c + \frac{\mu\gamma_0}{4}$, we find that $0$ is not a minimizer of $\Fcal_{p}$ by \eqref{eq:fpr2}. Hence, we conclude
\[
\argmin_{\Pcal} \Fcal_p = \left\{\gamma_0 e_2 \mathbbm{1}_{B_1(-2e_1)},\gamma_0 e_2 \mathbbm{1}_{B_1(2ke_1)}\right\},
\]
which finishes the proof. \medskip

\textit{Case 2: $\Lambda = \Lambda_t$.} For the triangular grid, the only difference is to estimate the stress at different locations. It is enough to consider the points $x=(2m-1,(2n-1)\sqrt{3}) \in \Lambda_t$ with $m \in \{0,\ldots,k\}$ and $n \in \Z$, since all the other points can be handled as in one of the cases of the square grid. We compute
\[
\frac{\sigma(p)(x)}{\mu}= (0,\gamma) + \frac{\gamma_0}{2} \left(\sum_{j=0}^{k-1} \frac{-2(2(m-j)-1)(2n-1)\sqrt{3}}{((2(m-j)-1)^2+3(2n-1)^2)^2},\sum_{j=0}^{k-1}\frac{(2(m-j)-1)^2-3(2n-1)^2}{((2(m-j)-1)^2+3(2n-1)^2)^2}\right).
\]
Analogously to before, the  sum in the first component is estimated by
\begin{align}
\begin{split}\label{eq:comp1}
\abs*{\sum_{j=0}^{k-1} \frac{-2(2(m-j)-1)(2n-1)\sqrt{3}}{((2(m-j)-1)^2+3(2n-1)^2)^2}} &\leq 2\sum_{j=0}^{\infty}\frac{(2j+1)(2n-1)\sqrt{3}}{((2j+1)^2+3(2n-1)^2)^2} \\
&\leq \frac{3\sqrt{3}}{8} \sum_{j=0}^{\infty}\frac{1}{(2j+1)^2}=\frac{3\sqrt{3}\pi^2}{64},
\end{split}
\end{align}
by maximizing each term over $n$. The  sum in the second component can be estimated from below by
\begin{equation}\label{eq:secondcomponent}
\sum_{j=0}^{k-1}\frac{(2(m-j)-1)^2-3(2n-1)^2}{((2(m-j)-1)^2+3(2n-1)^2)^2} \geq -\frac{1}{8} \sum_{j\in \Z} \frac{1}{(2j+1)^2} = -\frac{\pi^2}{32}.
\end{equation}
From above it can be estimated using Lemma \ref{lem:sumleq0},
\begin{align}
\begin{split}\label{eq:comp2}
\sum_{j=0}^{k-1}\frac{(2(m-j)-1)^2-3(2n-1)^2}{((2(m-j)-1)^2+3(2n-1)^2)^2} 
\leq \frac{\pi^2}{4\cosh^2(\pi \sqrt{3} (2n-1)/2)} \leq \frac{\pi^2}{4\cosh^2(\pi\sqrt{3}/2)}\,.
\end{split}
\end{align}
Combining all of the above, we find with $t:=\gamma_0/\gamma \in (0,1]$
\begin{equation}\label{eq:convexest}
\frac{|\sigma(p)(x)|-|\sigma(p)(-2e_1)|}{\mu\gamma_0} \leq t^{-1}\left(\left(\left(1+\frac{\pi^2}{8\cosh^2(\pi\sqrt{3}/2)}t\right)^2+\left(\frac{3\sqrt{3}\pi^2}{128}\right)^2t^2\right)^{1/2} - 1 - \frac{1}{8}t\right)\,.
\end{equation}
The right hand side can be continuously extended to $t=0$, where it is stricly negative, as it is in $t=1$. The numerator of the fraction on the right hand is convex in $t$, which shows that there exists $c>0$ such that
\[
|\sigma(p)(-2e_1)|-|\sigma(p)(x)|\geq c\mu\gamma_0\,.
\]
\end{proof}
\begin{remark}
\label{rem:unifestimate1}
In the proof, we have shown slightly more than necessary for later reference. Namely, we have shown that there exists $c>0$ such that for every 
 \[
p=\gamma_0\sum_{j=0}^{k-1} e_2\mathbbm{1}_{B_1(2j e_1)},
\]
and for every $x\in \Lambda\setminus \{-2e_1,2ke_1\}$, 
we have that
\[
|\sigma(p)(-2e_1)|-|\sigma(p)(x)|\geq c\mu\gamma_0\,.
\]
This holds true both for the quadratic and the triangular lattice.
\end{remark}

\begin{remark}[Other shear directions]
After rotating, it is clear that similar results to Theorem~\ref{th:shearband} hold when $g(x)=\gamma x \cdot \nu$ where 
\begin{itemize}
\item[(i)] $\nu \in \{\pm e_1, \pm e_2\}$ when $\Lambda=\Lambda_s$,
\item[(ii)] $\nu \in \{\pm e_2,\pm(\sqrt{3}/2,1/2),\pm(-\sqrt{3}/2,1/2)\}$ when $\Lambda=\Lambda_t$. \qedhere
\end{itemize}
\end{remark}

\subsection{Thickening of shear bands}

Due to the unbounded energy in the case $\Omega=\R^2$, the sets $\Acal_k(0)$ never stabilize, and, formally speaking, a full band of plastic zones is formed of infinite length. For this reason, we consider that the plasticity is given by this full band, and investigate what happens afterwards. It turns out that there is a clear dichotomy between the two different lattices. Indeed, for the square lattice, it is most favorable to activate a zone as far away from the band as possible, which, due to the unbounded domain, is not possible to resolve. For the triangular lattice, the next plastic events do activate next to the band, and subsequently form a full band next to it. In fact, this process continues like this, with bands forming next to each other and thus forming an infinitely thick band. These results show, on the one hand, that the location of the zones plays a big role, which corresponds to the mesh dependence in finite element simulations, while, on the other hand, that there is no intrinsic length scale for the thickness of the shear band arising from this model. We formulate the results separately, starting with the square lattice.

\begin{proposition}
\label{prop:no_fattening_quad_lattice}
Let $\Omega =\R^2$, and the boundary condition at infinity given by $g(x)=\gamma x_2$ with
\[
\gamma = \frac{\sigma_{c}}{\mu}+\frac{\gamma_0}{4} \geq \gamma_0.
\]
Suppose the admissible increments $\Ical$ are given by \eqref{eq:increments} with $\Lambda =\Lambda_s$. Then, we have with
\[
p = \gamma_0\sum_{j\in \Z} e_2\mathbbm{1}_{B_\delta(x_0+2j\delta e_1)}
\quad \text{that} \quad
\Acal(p) =\Acal_k(p)= \{p\} \quad \text{for all $k \in \N$}.
\]
\end{proposition}
\begin{proof}
Suppose without loss of generality that $x_0=0$ and $\delta=1$. Then, we compute with \eqref{eq:stressr2} that
\[
\frac{\sigma(p)(x)}{\mu}=\gamma e_2 +\gamma_0 \sum_{j \in \Z} \nabla v(x-2j e_1)-e_2\mathbbm{1}_{B_1(2j e_1)}(x).
\]
If $x =2je_1 \in \Lambda_s$ for $j \in \Z$ is an activated zone, then we find that
\[
\frac{\sigma(p)(x)}{\mu} = e_2 \left(\gamma-\gamma_0 \left(\frac{1}{2}-\frac{\pi^2}{24}\right)\right),
\]
and thus
\[
|\sigma(p)(x)| < \gamma \mu = \sigma_{c}+\frac{\gamma_0\mu}{4},
\]
meaning that \eqref{eq:minimizerschar} is not satisfied. On the other hand, if $x=(2m,2n) \in \Lambda_s$ with $n \not =0$, then we also find that the first component of the stress vanishes due to symmetry of the plastic band, cf.~\eqref{eq:eshelbystress}. Thus, we find using Lemma~\ref{le:sums} that
\[
\frac{\sigma(p)(x)}{\mu} = e_2 \left(\gamma -\gamma_0\frac{\pi^2}{8\sinh(\pi n)^2}\right).
\]
Again, we have that
\[
|\sigma(p)(x)| < \gamma \mu = \sigma_{c}+\frac{\gamma_0\mu}{4}.
\]
Hence, we find that the stress is everywhere below the critical threshold $\sigma_c+\gamma_0\mu/4$, which implies that zero is the only admissible plastic increment. This finishes the proof.
\end{proof}

We now proceed with the result for the triangular lattice, which states that if $\ell \in \N$ plastic bands are fully formed, then the next plastic events activate adjacent to the bands also aligned in a partial band.
\begin{theorem}
\label{thm:thickening_bands}
Let $\Omega =\R^2$, and the boundary condition at infinity given by $g(x)=\gamma x_2$ with
\[
\gamma = \frac{\sigma_{c}}{\mu}+\frac{\gamma_0}{4} \geq \gamma_0.
\]
Suppose the admissible increments $\Ical$ are given by \eqref{eq:increments} with $\Lambda =\Lambda_t$. Then, with plasticity given by
\[
p = \gamma_0\sum_{l=0}^{\ell-1}\sum_{j\in \Z} e_2\mathbbm{1}_{B_\delta(x_0+l\delta(1,\sqrt{3})+2j\delta e_1)} \quad \text{for $\ell \in \N$,}
\]
we have that
\[
\Acal_k(p)= \left\{p+ \gamma_0\sum_{j=0}^{k-1} e_2\mathbbm{1}_{B_\delta(x_0+l\delta(1,\sqrt{3})+2(j+m)\delta e_1)} \ \Big| \ m \in \Z, l \in \{-1,\ell\}\right\} \quad \text{for all $k \in \N$}.
\]
\end{theorem}
\begin{proof}
As before, we assume without loss of generality that $\delta=1$ and $x_0=0$. Then, assuming  $x=l(1,\sqrt{3})+2m e_1 \in \Lambda_t$, we compute with \eqref{eq:stressr2} that
\[
\frac{\sigma(p)(x)}{\mu}=\gamma e_2 +\gamma_0 \sum_{n=0}^{\ell-1}\sum_{j \in \Z} \nabla v(x-n(1,\sqrt{3})-2j e_1)-e_2\mathbbm{1}_{B_1(n(1,\sqrt{3})+2j e_1)}(x).
\]
From now on we use the notation 
\[
f(n):= \frac{1}{2}\sum_{j \in \Z} \frac{(n-2j)^2-3n^2}{((n-2j)^2+3n^2)^2}
\]
for $n\in\N$.
If  $l \in \{0,\ldots, \ell-1\}$, i.e., $x$ is the center of an activated zone, we find with this notation  that 
\[
\frac{\sigma(p)(x)}{\mu} = e_2 \left( \gamma -\gamma_0\left(\frac{1}{2} -\frac{\pi^2}{24}-\sum_{n=1}^{\ell-l-1}f(n)-\sum_{n=1}^{l} f(n) \right)\right)\,,
\]
where we have used that again the stress cancels in the first component due to symmetry. Now we use Lemma~\ref{le:sums2} from the appendix to simplify this expression: Combining this lemma with the fact that
$f(n)$ alternates between positive and negative values that decrease in absolute value, we infer that
\[
\frac{|\sigma(p)(x)|}{\mu}=\gamma -\gamma_0\left(\frac{1}{2} -\frac{\pi^2}{24}-\sum_{n=1}^{\ell-l-1}f(n)-\sum_{n=1}^{l} f(n)\right) \leq \gamma - \gamma_0 \left(\frac{1}{2}-\frac{\pi^2}{24}-2f(1)\right) < \gamma.
\]
On the other hand, if $x=l(1,\sqrt{3})+2me_1 \in \Lambda_t$ with $l \not \in \{0,\ldots \ell-1\}$ and set $l_0:=\min\{|l-(\ell-1)|,|l|\} \in \N$, then we find
\[
\frac{\sigma(p)(x)}{\mu} = e_2 \left( \gamma +\gamma_0 \sum_{n=0}^{\ell-1} f(n+l_0)\right).
\]
By equation \eqref{eq:monotone} of Lemma \ref{le:sums2}, we deduce that this quantity is maximized when $l_0=1$, which yields
\[
\frac{|\sigma(p)(x)|}{\mu}=e_2 \left( \gamma +\gamma_0 \sum_{n=1}^{\ell} f(n)\right) > \gamma.
\]
Hence, we deduce that
\[
\max_{x \in \Lambda_t} |\sigma(p)(x)| > \sigma_c + \frac{\gamma_0\mu}{4} \quad \text{with} \quad \argmax_{x \in \Lambda_t}|\sigma(p)(x)| = \left\{l(1,\sqrt{3})+2me_1 \,|\, m \in \Z, \ l \in \{-1,\ell\}\right\}.
\]
By \eqref{eq:minimizerschar} we obtain that
\[
\Acal_1(p) = \left\{p + \gamma_0e_2\mathbbm{1}_{B_1(l(1,\sqrt{3})+2me_1)}\,|\, m \in \Z, \ l \in \{-1,\ell\}\right\}.
\]

Next, we proceed by induction and assume that for some $k \in \N$
\[
\Acal_k(p) = \left\{p+ \gamma_0\sum_{j=0}^{k-1} e_2\mathbbm{1}_{B_1(x_0+l(1,\sqrt{3})+2(j+m) e_1)} \ \Big| \ m \in \Z, l \in \{-1,\ell\}\right\}.
\]
To determine $\Acal_{k+1}(p)$, we take $p' \in \Acal_{k}(p)$. By translating and reflecting, we may assume that
\[
p'=p+\gamma_0\sum_{j=0}^{k-1} e_2\mathbbm{1}_{B_1(x_0+\ell(1,\sqrt{3})+2j e_1)}.
\]
Using \eqref{eq:stressr2}, we obtain
\begin{align*}
\frac{\sigma(p')(x)}{\mu}&=\gamma e_2 +\gamma_0 \sum_{l=0}^{\ell-1}\sum_{j \in \Z} \nabla v(x-l(1,\sqrt{3})-2j e_1)-e_2\mathbbm{1}_{B_1(l(1,\sqrt{3})+2j e_1)}(x)\\
&\quad+\gamma_0\sum_{j=0}^{k-1}\nabla v(x-\ell(1,\sqrt{3})-2je_1)-e_2\mathbbm{1}_{B_1(\ell(1,\sqrt{3})+2je_1)}(x).
\end{align*}
If $x \in \Lambda_t$ is the center of a non-activated zone, we may use that the full bands do not contribute to the first component of the stress to find 
\[
\frac{\sigma(p')(x)}{\mu}=\gamma e_2  +\gamma_0 \sum_{j=0}^{k-1} \nabla v(x-\ell(1,\sqrt{3})-2j e_1) +\gamma_0 e_2\sum_{n=0}^{\ell-1} f(n+l_0)
\]
with $l_0=\min\{|l-(\ell-1)|,|l|\}$. The first two terms are the same as in the case without bands, while the last term is due to these fully formed bands. By \eqref{eq:secondcomponent} and \eqref{eq:monotone} together with $\gamma \geq \gamma_0$, it is readily seen that the second component of the stress never becomes negative. Therefore, by using that the first two terms are maximized at $x_{+}:=\ell(1,\sqrt{3})+2ke_1$ and $x_{-}:=\ell(1,\sqrt{3})-2e_1$ by Theorem~\ref{th:shearband}, and the last term is as well due to \eqref{eq:monotone}, we infer that among all non-activated zones, the maximal stress is
\[
|\sigma(p')(x_{\pm})| = \mu \left(\gamma + \gamma_0\left(\frac{1}{8}\sum_{j=1}^{k}\frac{1}{j^2}+\sum_{n=1}^{\ell} f(n)\right)\right) > \mu\gamma = \sigma_c + \frac{\gamma_0\mu}{4};
\]
note, in particular, that the stress exceeds the threshold for activation here. It remains to compare this with the stress at the already activated zones. If $x = \ell(1,\sqrt{3})+2me_1 \in \Lambda_t$ with $m \in \{0,\ldots,k-1\}$ is part of the partial band, then we can just argue as in Theorem~\ref{th:shearband}; indeed, the only addition is a positive contribution in the second component of the stress of the form
\[
\gamma_0\sum_{n=1}^{\ell} f(n),
\]
but since this is also there at $x_{\pm}$ the same estimates still hold. If $x=l(1,\sqrt{3})+2m e_1 \in \Lambda_t$ with $l \in \{0,\ldots \ell-1\}$ is part of a full shear band, we find
\[
\frac{\sigma(p')(x)}{\mu} = e_2 \left( \gamma -\gamma_0\left(\frac{1}{2} -\frac{\pi^2}{24}-\sum_{n=1}^{\ell-l-1}f(n)-\sum_{n=1}^{l-1} f(n) \right)\right)+\gamma_0 \sum_{j=0}^{k-1} \nabla v(x-\ell(1,\sqrt{3})-2j e_1).
\]
For similar reasons, the second component of the stress does not become negative. Using \eqref{eq:comp1} and \eqref{eq:comp2} to estimate the gradients of $v$, and \eqref{eq:monotone} again, we find
\begin{align*}
\frac{|\sigma(p')(x)|}{\mu} &\leq \left(\left(\gamma-\gamma_0\left(\frac{1}{2} -\frac{\pi^2}{24}-3f(1)\right)\right)^2+\left(\frac{3\sqrt{3}\pi^2}{128}\right)^2\gamma_0^2\right)^{1/2} \\
&\leq \left(\gamma^2+\left(\frac{3\sqrt{3}\pi^2}{128}\right)^2\gamma_0^2\right)^{1/2} < \gamma + \frac{1}{8}\gamma_0 \leq \frac{|\sigma(p')(x_{\pm})|}{\mu},
\end{align*}
see~\eqref{eq:convexest}. This shows that
\[
\argmax_{x \in \Lambda_t} |\sigma(p')(x)| = \{x_{+},x_{-}\} \quad \text{with} \quad \max_{x \in \Lambda_t} |\sigma(p')(x)| = |\sigma(p')(x_{\pm})|>\sigma_c + \frac{\gamma_0\mu}{4}.
\]
By \eqref{eq:minimizerschar} this finishes the proof.
\end{proof}

\section{Formation of shear bands in half space}

\label{sec:half-space}

The results in the previous section show that plastic events induce cascades that form banded structures, at least when considering settings that are well approximated by an unbounded domain. In this section we investigate the effect that a boundary may have on these results. Specifically, we consider a half space as our domain with Neumann boundary conditions on the boundary. We show that also in this case a shear band forms instantaneously, and that this shear band will initiate from the boundary. This indicates that in bounded domains it is more likely that shear bands form near the boundary, which is a good reflection of the physical situation.

The precise set-up is as follows. We consider the domain $\R^{2}_{+} = (0,\infty) \times \R$ and the square lattice
\begin{equation}\label{eq:latticeplus}
\Lambda_{+}:= \{ (2m+1,2n) \,:\, m \in \Z_{\geq 0}, n \in \Z\},
\end{equation}
where we have offset the grid to ensure all zones are contained in $\R^{2}_{+}$. We consider again the linear boundary condition
\[
g(x) = \gamma x_2
\]
at infinity, while we impose natural boundary conditions on $\{0\} \times \R$. In this setting, the solution $w_{\nu,y} \in H^1(\R^2_{+})$ of \eqref{eq:displacementupdate} is given by
\[
w_{\nu,y}(x) = v_\nu(x-y)+v_{\nu}(x_1+y_1,x_2-y_2),
\]
with $v_\nu$ as in \eqref{eq:displcamentr2}; this corresponds to the sum of two full space solutions, centered at $y$ and its reflection through the boundary $(-y_1,y_2)$. From \eqref{eq:quadraticcorr}, we infer that
\begin{equation*}\label{eq:fprplus}
\Fcal_p(q_{\nu,y})=-\gamma_0|B_\delta(0)|\left( \sigma(p)(y) \cdot \nu -\sigma_c - \gamma_0\mu\left(\frac{1}{4}+\frac{\delta^2}{16y_1^2}\left(\nu_1^2-\nu_2^2\right)\right) \right).
\end{equation*}
This shows that $q_{\nu,y} \in \argmin_{\Pcal} \Fcal_p$ if and only if
\begin{equation}\label{eq:mincharplus}
\sigma(p)(y) \cdot \nu + \gamma_0 \frac{\mu \delta^2}{16y_1^2}\left(-\nu_1^2+\nu_2^2\right) = \max_{(x,\theta) \in \delta\Lambda_{+}\times\S^1} \sigma(p)(x) \cdot \theta+ \gamma_0 \frac{\mu \delta^2}{16x_1^2}\left(-\theta_1^2+\theta_2^2\right) \geq \sigma_c + \frac{\mu \gamma_0}{4}.
\end{equation}
We have the following result.
\begin{theorem}\label{th:shearneumann}
Let $\Omega =\R^2_{+}$, and the boundary condition at infinity given by $g(x)=\gamma x_2$ with
\[
\gamma = \frac{\sigma_{c}}{\mu}+\frac{3\gamma_0}{16} \geq \gamma_0.
\]
Suppose the admissible increments $\Ical$ are given by \eqref{eq:increments} with $\Lambda =\Lambda_{+}$.  Then, it holds that
\[
\Acal_{k}(0) = \{0\} \cup \left\{ \gamma_0\sum_{j=0}^{k-1} e_2\mathbbm{1}_{B_\delta(\delta(1,2n)+2j\delta e_1)} \ \Big| \ \text{$n \in \Z$}\right\} \quad \text{for any $k \in \N$}.
\]
\end{theorem}
\begin{proof}
As before, we assume that $\delta=1$. Given the initial plasticity $p_0=0$, we find that the stress $\sigma(p_0) \equiv \mu \gamma e_2$ is uniform. From this, it is clear that \eqref{eq:mincharplus} is maximized at $y=(1,2n)$ for $n \in \Z$ and $\nu =e_2$, leading to
\[
 \sigma(p)(y) \cdot \nu + \gamma_0 \frac{\mu}{16y_1^2}\left(-\nu_1^2+\nu_2^2\right) = \mu\left(\gamma + \frac{\gamma_0}{16}\right) = \sigma_c + \frac{\mu \gamma_0}{4}.
\]
Due to the equality, zero is also an admissible increment, and we find
\[
\Acal_1(0)=\{0\} \cup \left\{ \gamma_0 e_2 \mathbbm{1}_{B_1((1,2n))} \,|\, n \in \Z \right\}.
\]
We proceed by induction and suppose that
\[
\Acal_k(0) = \{0\} \cup \left\{ \gamma_0\sum_{j=0}^{k-1} e_2\mathbbm{1}_{B_1((1,2n)+2j e_1)} \ \Big| \ \text{$n \in \Z$}\right\}
\]
for some $k \in \N$. Then, we take $p \in \Acal_{k}(0) \setminus \{0\}$, which, by translating, may be assumed to be given by
\[
p = \gamma_0\sum_{j=0}^{k-1} e_2\mathbbm{1}_{B_1((2j+1)e_1)}.
\]
We then find that for $x\in\Omega$, the stress is given by
  \begin{equation}\label{eq:7}
  \begin{split}
\frac{\sigma(p)(x)}{\mu}&=\gamma e_2 + \gamma_0 \sum_{j=0}^{k-1} \nabla v (x-(2j+1)e_1) + \nabla v(x+(2j+1)e_1)-\mathbbm{1}_{B_1((2j+1)e_1)}(x)\\
&=\gamma e_2 + \gamma_0 \sum_{j=-k}^{k-1} \nabla v (x-(2j+1)e_1) -\mathbbm{1}_{B_1((2j+1)e_1)}(x)\,.
\end{split}
\end{equation}
Our goal is to show that
\begin{equation}\label{eq:goalargmax}
\argmax_{(x,\theta) \in \Lambda_{+}\times\S^1} \sigma(p)(x) \cdot \theta+ \gamma_0 \frac{\mu}{16x_1^2}\left(-\theta_1^2+\theta_2^2\right) = \{((2k+1)e_1,e_2)\}.
\end{equation}
If $x = (2m+1,2n) \in \Lambda_{+}$ for $m \geq k$ and $n \in \Z$, and $\nu \in \S^1$ then we find that
\begin{align*}
\sigma(p)(x)\cdot \nu + \gamma_0 \frac{\mu}{16x_1^2}\left(-\nu_1^2+\nu_2^2\right) &\leq \sigma(p)((2m+1)e_1)\cdot e_2+\gamma_0 \frac{\mu}{16(2m+1)^2} \\
&\leq \sigma(p)((2k+1)e_1)\cdot e_2+\gamma_0 \frac{\mu}{16(2k+1)^2},
\end{align*}
where in the first inequality we used that when moving from $x$ to $(2m+1)e_1$ all the gradients of $v$ become larger in length and align with the $e_2$ direction, cf.~\eqref{eq:eshelbystress}, while for the second inequality we used that both terms become larger. We note  that at least one of the inequalities is strict if $x \not = (2k+1)e_1$ or $\nu \not = e_2$; more precisely there exists $c>0$ such that
\[
\frac{|\sigma(p)((2k+1)e_1)|}{\mu}+\frac{\gamma_0}{16(2k+1)^2}-\frac{|\sigma(p)(x)|}{\mu}-\frac{\gamma_0}{16|x_1|^2}\geq c\gamma_0
\]
for all $x=(2m+1,2n)$ that satisfy either $m>k$ or $m=k$ and $n\neq 0$. 

\medskip

Next, suppose that $x = (2m+1,2n) \in \Lambda_{+}$ with $m \in \{0,\ldots,k-1\}$ and $n \not =0$. We want to estimate the modulus of the right hand side in \eqref{eq:7} for such an $x$. This may be done as in the estimates for the analogous term in the proof of Theorem~\ref{th:shearband}, see the estimates \eqref{eq:comp1estimate}, \eqref{eq:lowercomp2estimate}, and \eqref{eq:uppercomp2estimate}, which go through unchanged by the fact that the range of summation has changed from $\sum_{j=0}^{k-1}$ there to $\sum_{j=-k}^{k-1}$ here.
In this manner we obtain 
\[
\frac{|\sigma(p)(x)|}{\mu}\leq \gamma \sqrt{1+a_s^2t^2}
\]
with $a_s:=\frac{\sqrt{3}\pi^2}{128}$, $t=\frac{\gamma_0}{\gamma}\leq 1$.
Thus 
\[
  \begin{split}
|\sigma(p)(x)|+\frac{\mu\gamma_0}{16|x_1|^2}-|\sigma(p)((2k+1)e_1)|-\frac{\mu\gamma_0}{16(2k+1)^2}&\leq \mu\gamma_0\left(\frac{\sqrt{1+a_s^2 t^2}-1-\frac{1}{8}t}{t}+\frac{1}{16}\right)\\
&=\mu\gamma_0\frac{\sqrt{1+a_s^2 t^2}-1-\frac{1}{16}t}{t}\,.
\end{split}
\]
 The function on the right extends continuously to $t=0$ with value
 $-\mu\gamma_0/16$ and is strictly negative  on $[0,1]$ (since the numerator is convex, and negative in $\{0,1\}$). Therefore
we have shown that for all $x=(2m+1,2n)$ with $m\in\{0,\dots,k-1\}$, $n\neq 0$,
\[
\frac{|\sigma(p)((2k+1)e_1)|}{\mu}+\frac{\gamma_0}{16(2k+1)^2}-\frac{|\sigma(p)(x)|}{\mu}-\frac{\gamma_0}{16|x_1|^2}{\mu}\geq c\gamma_0\,,
\]
for $c>0$ chosen small enough.

Finally, if $x = (2m+1,0) \in \Lambda_{+}$ with $m \in \{0,\ldots,k-1\}$ is an activated zone, we find as in Theorem~\ref{th:shearband} that
\begin{align*}
\sigma(p)(x) \cdot \nu+\gamma_0 \frac{\mu}{16x_1^2}\left(-\nu_1^2+\nu_2^2\right) \leq \mu \left(\gamma-\frac{\gamma_0}{2}+\frac{\gamma_0\pi^2}{24}\right)  +\gamma_0 \frac{\mu}{16} < \mu\left( \gamma + \frac{\gamma_0}{8}\right).
\end{align*}
This proves \eqref{eq:goalargmax}, and since we further have
\[
\sigma(p)((2k+1)e_1)\cdot e_2+\gamma_0 \frac{\mu}{16(2k+1)^2} \geq \mu\left( \gamma + \frac{\gamma_0}{8}\right) > \frac{\sigma_{c}}{\mu}+\frac{\gamma_0}{4},
\]
we find by \eqref{eq:mincharplus} that
\[
\argmin_{\Pcal} \Fcal_p = \{\gamma_0 e_2\mathbbm{1}_{B_1((2k+1)e_1)}\},
\]
as desired.
\end{proof}

\begin{remark}
\label{rem:unifestimate2}
  We note for later reference that once more  we have proven slightly more than required in the theorem statement: Namely, there exists $c>0$ such that for 
\[
p = \gamma_0\sum_{j=0}^{k-1} e_2\mathbbm{1}_{B_1((2j+1)e_1)}
\]
and any $x\in \Lambda_+\setminus\{(2k+1,0)\}$, we have that
\[
\frac{|\sigma(p)((2k+1)e_1)|}{\mu}+\frac{\gamma_0}{16(2k+1)^2}-\frac{|\sigma(p)(x)|}{\mu}-\frac{\gamma_0}{16|x_1|^2}{\mu}\geq c\gamma_0\,.
\]
\end{remark}
\section{Quantitative estimates for large bounded domains}

\label{sec:bounded_domains}

In this section, we will show that the formation of shear bands also occurs on bounded domains, with an explicit guarantee on the scaling between the length of the shear band and size of the domain. While there are no explicit formulas available in this case, we are able to use quantitative estimates on the difference with the unbounded domain solutions to obtain these statements. The main result is that on bounded domains the length of the shear band will be at least proportional to the ratio of the size of the domain and the size of the zones.

\subsection{Dirichlet boundary conditions}

\label{sec:dirichlet}
We detail the set-up. Consider a bounded Lipschitz domain $\Omega \subset \R^2$ and assume that we have Dirichlet boundary conditions on the entire boundary, so $\Gamma=\partial\O$ with boundary condition $g(x)=\gamma x_2$. The possible activation sites for the plastic zones are the square lattice or hexagonal lattice $\Lambda \in \{\Lambda_s,\Lambda_t\}$ with the zones entirely contained in $\O$ as in \eqref{eq:increments}.

In this case the solution to \eqref{eq:displacementupdate} is not known, but we can write it as $w_{\nu,y}(x)=v_\nu(x-y)+\tilde{v}_{\nu,y}(x)$ where $v_\nu$ is as in \eqref{eq:displcamentr2} and $\tilde{v}_{\nu,y}$ is the harmonic function solving
\begin{equation*}
\begin{cases}
\Delta \tilde{v}_{\nu,y} = 0 &\text{in $\O$,}\\
\tilde{v}_{\nu,y}(x) = -v_{\nu}(x-y) &\text{on $\partial\O$}.
\end{cases}
\end{equation*}
In fact, since $v_\nu(x)=\nu \cdot V(x)$ with
\[
V(x):=\begin{cases}
\dfrac{1}{2} x &\text{for $x \in B_\delta(0)$},\\[2ex]
\dfrac{\delta^2}{2} \dfrac{x}{|x|^2} &\text{for $x \in B_\delta(0)^c$},
\end{cases}
\]
we also find that $\tilde{v}_{\nu,y}(x)=\nu \cdot V_y(x)$, where $V_y \in H^1(\Omega;\R^2)$ solves
\begin{equation}\label{eq:diricorr}
\begin{cases}
\Delta V_{y} = 0 &\text{in $\O$,}\\
V_{y}(x) = -V(x-y) &\text{on $\partial\O$}.
\end{cases}
\end{equation}
Additionally, the term $\Jcal_0(\mathbbm{1}_{B_\delta(y)}\nu)$ in the change in energy formula \eqref{eq:quadraticcorr} reduces to 
\begin{equation}
\begin{split}\label{eq:Iform}
\Jcal_0(\mathbbm{1}_{B_\delta(y)}\nu) = \frac{\mu}{4}\abs{B_\delta(0)} - \frac{\mu}{2}\abs{B_\delta(0)} \nu \cdot \nabla V_{y}(y)\nu
\end{split}
\end{equation}
We obtain the following estimates on $V_{y}$.
\begin{lemma}
For any $\xi\in\S^1$, it holds that
\begin{equation}\label{eq:corrbound}
\abs{\xi\cdot \nabla V_{y}(x)} \leq \frac{\delta^2}{\dist(x,\partial\O)\dist(y,\partial \Omega)} \quad \text{for all $x \in \O$},
\end{equation}
and, therefore,
\begin{equation}\label{eq:enbound}
\frac{\mu}{4}\abs{B_\delta(0)} \leq \Jcal_0(\mathbbm{1}_{B_\delta(y)}\nu) \leq \frac{\mu}{4}\abs{B_\delta(0)}+ \frac{\mu \delta^2}{2\dist(y,\partial \O)^2}\abs{B_\delta(0)}.
\end{equation}
\end{lemma}
\begin{proof}
Using harmonicity of $\nabla V_y$ on the ball $B_r(x)$ with $r=\dist(x,\partial\O)$, we have that
\[
\nabla V_y(x)=\frac{1}{\pi r^2}\int_{ B_r(x)}\nabla V_y(z)\d z=\frac{2}{  r}\fint_{\partial B_r(x)}V_y(s)\otimes n(s)\d \mathcal H^1(s)
\]
and hence (using the maximum principle for  $\xi\cdot V_y$) 
\[
\abs{\xi\cdot \nabla V_{y}(x)} \leq \frac{2}{\dist(x,\partial \O)} \norm{V_{y}}_{L^\infty(\partial\O)}\,.
\]
Since $V_{y}(x)=-V(x-y)$ on $\partial\O$, we find that
\[
\norm{V_{y}}_{L^\infty(\partial\O)} = \max_{x \in \partial \O} \abs{V(x-y)} \leq \max_{x \in \partial \O}\frac{\delta^2}{2} \frac{1}{\abs{x-y}} = \frac{\delta^2}{2}\dist(y,\partial\O)^{-1},
\]
which proves \eqref{eq:corrbound}. To show \eqref{eq:enbound}, we first observe that clearly
\[
\Jcal_0(\mathbbm{1}_{B_\delta(y)}\nu) \geq \min \left\{ \int_{\R^2}\frac{\mu}{2}\abs{\nabla w - \nu\mathbbm{1}_{B_{\delta}(y)}}^2\dd x \,:\, w \in H^1(\R^2)\right\} = \frac{\mu}{4}\abs{B_\delta(0)}.
\]
On the other hand, applying \eqref{eq:corrbound} in \eqref{eq:Iform} immediately yields
\[
\Jcal_0(\mathbbm{1}_{B_\delta(y)}\nu) \leq \frac{\mu}{4}\abs{B_\delta(0)}+ \frac{\mu \delta^2}{{2}\dist(y,\partial \O)^2}\abs{B_\delta(0)}
\]
as desired.
\end{proof}

We now have the following result, which states that a partial shear band forms whose length is proportional to the size of the domain divided by $\delta$.

\begin{theorem}\label{th:partialshearband}
Let $\Omega \subset \R^2$ be a bounded Lipschitz domain with $0 \in \Omega$ and such that 
\begin{equation}\label{eq:symmetry}
\Omega = \{x \in \R^2 \,|\, (x_1,-x_2) \in \Omega\}.
\end{equation}
Suppose the increments $\Ical$ are given by \eqref{eq:increments} with $\Lambda \in \{\Lambda_s,\Lambda_t\}$ and consider the initial plasticity $p=\gamma_0 e_2 \mathbbm{1}_{B_\delta(0)}$. If the boundary condition is given by $g(x)=\gamma x_2$ with
\begin{equation}\label{eq:stabilityinit}
\gamma = \frac{\sigma_c}{\mu}+\min_{x = \pm 2\delta e_1}\gamma_0\left(\frac{\Jcal_0(e_2\mathbbm{1}_{B_\delta(x)})}{\mu|B_\delta(0)|}-\nabla w_{e_2,0}(x)\cdot e_2\right) \geq \gamma_0,
\end{equation}
then, there exists a constant $\alpha >0$ independent of $\gamma_0,\gamma,\delta,\Omega$ such that with $L:=\dist(0,\partial \Omega)/\delta$ it holds for $k \in \N$ that
\[
\{p\} \subsetneq \Acal_{k}(p) \subset \{p\} \cup \left\{\gamma_0\sum_{j=0}^{k} e_2\mathbbm{1}_{B_\delta(2(j-m)\delta e_1)} \,\Big|\, m \in \{0,\ldots,k\}\right\} \quad \text{if  $k \leq \alpha L$}.
\]
\end{theorem}
\begin{proof}
By rescaling the domain, we may assume that $\delta=1$ and $L=\dist(0,\partial\O)$.  Suppose that  the plasticity is given by
\[
p' = \gamma_0\sum_{j=0}^{k-1} e_2\mathbbm{1}_{B_1(2(j-m) e_1)},
\]
for some $m \in \{0,\ldots,k-1\}$. For notational simplicity, we assume $m=0$. The general case is identical,
since translating the shear band only changes the distance of its endpoints
to the boundary by at most $2k$, and under the assumption
$k\le \alpha L$ all estimates below remain unchanged after possibly
decreasing $\alpha$. We infer from \eqref{eq:newdisplacementstress} that
\[
\frac{\sigma(p')(x)}{\mu}=\gamma e_2 + \gamma_0\sum_{j=0}^{k-1} \nabla v(x-2je_1) +\nabla V_{2je_1}(x)e_2 - e_2\mathbbm{1}_{B_1(2je_1)}(x).
\]
Let us denote the stress associated to the unbounded domain solution by
\[
\sigma_{\R^2}(p')(x) := \mu \left(\gamma e_2 +\gamma_0 \sum_{j=0}^{k-1} \left( \nabla v(x-2je_1) - e_2\mathbbm{1}_{B_1(2je_1)}(x)\right)\right)\,.
\]
By Remark \ref{rem:unifestimate1} there exists $c>0$ (independently of $k$) such that 
\[
\abs{\sigma_{\R^2}(p')(x_{\pm})} \geq \abs{\sigma_{\R^2}(p')(x)}+c\mu\gamma_0 \quad \text{for all $x \in \Lambda \setminus \{x_{\pm}\}$},
\]
with $x_{+}=2ke_1$ and $x_{-} = -2e_1$. Now, if $1 \leq k \leq L/4$, we infer from \eqref{eq:3} and \eqref{eq:corrbound} that
\begin{align}
\begin{split}\label{eq:stresscomp}
\abs{\sigma(p')(x_{\pm})} &\geq \abs{\sigma_{\R^2}(p')(x_{\pm})} - \mu\gamma_0 \sum_{j=0}^{k-1} \frac{1}{\dist(x_{\pm},\partial \Omega)\dist(2je_1,\partial \Omega)} \\
&\geq \abs{\sigma_{\R^2}(p')(x_{\pm})} -  \frac{4 \mu \gamma_0k}{L^2}\\
&\geq \mu\left(\gamma+\frac{\gamma_0}{8}\right)-  \frac{4  \mu \gamma_0k}{L^2}\,,
\end{split}
\end{align}
where we have assumed $\alpha<\frac14$, and hence  
\[
\min\{\dist(x_{\pm},\partial \Omega),\dist(2je_1,\partial \Omega)\} \geq L/2\,.
\]
If $\nu \in \S^1$, we find from \eqref{eq:changestz} and \eqref{eq:Iform}
  \begin{equation}\label{eq:6}
\Fcal_{p'}(\gamma_0\nu \mathbbm{1}_{B_1(x_{\pm})}) = -\gamma_0|B_1(0)| \left(\sigma(p')(x_{\pm}) \cdot \nu -\sigma_c -\frac{\gamma_0\mu}{4}+\frac{\gamma_0\mu}{2}\nu \cdot \nabla V_{x_{\pm}}(x_{\pm})\nu \right).
\end{equation}
Notice that by the symmetry of the domain \eqref{eq:symmetry}, we find that $\nabla V_{x_{\pm}}(x_{\pm})$ is a diagonal matrix. Again, by symmetry considerations, including the covariance of the boundary conditions $g(Rx)=-g(x)$ under the reflection $Rx=(x_1,-x_2)$, the stress $\sigma(p')(x_{\pm})$ is a multiple of $e_2$. 

Let $\sigma(p')(x_\pm)=s_\pm e_2$ and set
$A_\pm:=\nabla V_{x_\pm}(x_\pm)$. Since $A_\pm$ is diagonal, for every
$\nu\in\mathbb S^1$ we have
\[
\nu\cdot A_\pm\nu-e_2\cdot A_\pm e_2
=
\bigl((A_\pm)_{11}-(A_\pm)_{22}\bigr)\nu_1^2.
\]
Hence, using $\nu_1^2\leq 2(1-\nu_2)$,
\[
\begin{aligned}
&\Fcal_{p'}(\gamma_0\nu\mathbbm{1}_{B_1(x_\pm)})
-\Fcal_{p'}(\gamma_0e_2\mathbbm{1}_{B_1(x_\pm)})\\
&\qquad\geq
\gamma_0|B_1(0)|(1-\nu_2)
\left(
s_\pm-C\mu\gamma_0L^{-2}
\right).
\end{aligned}
\]
By the lower bound on $s_\pm$ from \eqref{eq:stresscomp}, the quantity in parentheses is positive
for $\alpha$ sufficiently small. Therefore the minimum over
$\nu\in\mathbb S^1$ is attained uniquely at $\nu=e_2$.
For $\nu=e_2$, we use \eqref{eq:stresscomp} in \eqref{eq:6} together with \eqref{eq:corrbound} to find
\begin{equation}\label{eq:Fpxpm}
\Fcal_{p'}(\gamma_0e_2 \mathbbm{1}_{B_1(x_{\pm})}) \leq -\gamma_0|B_1(0)| \left(|\sigma_{\R^2}(p')(x_{\pm})| -\sigma_c -\frac{\mu\gamma_0}{4}-\frac{4\mu\gamma_0(k+1)}{L^2} \right).
\end{equation}
On the other hand, for any other $x \in \Lambda \cap \Omega_1 \setminus \{x_{\pm}\}$, we have that $\dist(x,\partial \Omega) \geq 1$, so that \eqref{eq:corrbound} yields
\begin{align*}
\abs{\sigma(p')(x)} &\leq \abs{\sigma_{\R^2}(p')(x)} + \mu\gamma_0 \sum_{j=0}^{k-1} \frac{1}{\dist(x,\partial \Omega)\dist(2je_1,\partial \Omega)} \\
&\leq \abs{\sigma_{\R^2}(p')(x)} +  \frac{2\mu\gamma_0 k}{L}.
\end{align*}
Hence, we find by \eqref{eq:enbound} that
\begin{align*}
\Fcal_{p'}(\gamma_0\nu \mathbbm{1}_{B_1(x)}) &\geq -\gamma_0|B_1(0)| \left(|\sigma_{\R^2}(p')(x)| -\sigma_c - \frac{\mu\gamma_0}{4}+\frac{2\mu\gamma_0 k}{L} \right) \\
&\geq -\gamma_0|B_1(0)| \left(|\sigma_{\R^2}(p')(x_{\pm})| -\sigma_c - \frac{\mu\gamma_0}{4} +\frac{2\mu\gamma_0 k}{L}-c\mu\gamma_0\right) \\
&> \Fcal_{p'}(\gamma_0e_2 \mathbbm{1}_{B_1(x_{\pm})})
\end{align*}
for $\frac{k}{L}\leq\alpha$ chosen small enough.
This proves that either $\gamma_0e_2 \mathbbm{1}_{B_1(x_+)}$ or $\gamma_0e_2 \mathbbm{1}_{B_1(x_-)}$ must be the optimal non-zero plastic increment in that case; it could be both as well. It remains to check that this is energetically favorable when compared to not activating any plastic zone. When $k=1$, we find that
\[
\min_{x = x_{\pm}}\Fcal_{p'}(\gamma_0e_2 \mathbbm{1}_{B_1(x)})=0
\]
by \eqref{eq:stabilityinit}, which shows that both $0$ and at least one of $\gamma_0e_2 \mathbbm{1}_{B_1(x_{\pm})}$ are admissible increments. 

\medskip

To treat the case $k>1$, we first determine $\gamma$ more precisely. For
$x=\pm2e_1$, \eqref{eq:Iform} and the decomposition of
$w_{e_2,0}$ into the whole-space solution and its boundary correction give
\[
\frac{\Jcal_0(e_2\mathbbm{1}_{B_1(x)})}
     {\mu|B_1(0)|}
=
\frac14-\frac12 e_2\cdot\nabla V_x(x)e_2
\]
and
\[
\nabla w_{e_2,0}(x)\cdot e_2
=
\frac18+e_2\cdot\nabla V_0(x)e_2.
\]
It therefore follows from \eqref{eq:stabilityinit} and
\eqref{eq:corrbound} that
\[
\begin{aligned}
\gamma
&=
\frac{\sigma_c}{\mu}
+\min_{x=\pm2e_1}\gamma_0
\left(
\frac18
-\frac12 e_2\cdot\nabla V_x(x)e_2
-e_2\cdot\nabla V_0(x)e_2
\right)\\
&=
\frac{\sigma_c}{\mu}
+\frac{\gamma_0}{8}
+O\left(\frac{\gamma_0}{L^2}\right).
\end{aligned}
\]
On the other hand, the explicit whole-space computation (see \eqref{eq:3} from the proof of
Theorem~\ref{th:shearband}) yields
\[
|\sigma_{\R^2}(p')(x_\pm)|
=
\mu\left(
\gamma+\frac{\gamma_0}{8}
\sum_{j=1}^{k}\frac1{j^2}
\right).
\]
Consequently, for $k>1$,
\[
\begin{aligned}
|\sigma_{\R^2}(p')(x_\pm)|
-\sigma_c-\frac{\mu\gamma_0}{4}
&=
\mu\left(\gamma-\frac{\sigma_c}{\mu}-\frac{\gamma_0}{8}\right)
+\frac{\mu\gamma_0}{8}
\sum_{j=2}^{k}\frac1{j^2}\\
&\geq
\frac{\mu\gamma_0}{32}
-C\frac{\mu\gamma_0}{L^2},
\end{aligned}
\]
where we used
\[
\sum_{j=2}^{k}\frac1{j^2}\geq\frac14.
\]
Combining this estimate with \eqref{eq:Fpxpm}, we obtain
\[
\Fcal_{p'}(\gamma_0e_2\mathbbm{1}_{B_1(x_\pm)})
\leq
-\gamma_0|B_1(0)|
\left(
\frac{\mu\gamma_0}{32}
-C\frac{\mu\gamma_0}{L^2}
-\frac{4\mu\gamma_0 (k+1)}{L^2}
\right)<0
\]
whenever $k/L\leq\alpha$, after decreasing $\alpha>0$ if necessary.
Thus, at least one endpoint increment has strictly lower energy than the zero
increment for every $k>1$. This completes the proof.
\end{proof}
\begin{remark}[Non-symmetric domains]
It is clear by translation that the initial plasticity may be centered at a different point $x_0 \not =0$, as long as the symmetry condition \eqref{eq:symmetry} is satisfied around $x_0$. The symmetry condition is there to ensure that the plasticity still activates in the same direction as the unbounded domain case, which prevents the exponential accumulation of errors. With this in mind, the symmetry condition may be removed if the admissible increments also restrict the angles of the plasticity
\[
\Ical=\{\gamma_0\nu \mathbbm{1}_{B_\delta(y)} \,|\, \nu \in \{\nu_1,\ldots,\nu_\ell\}, y \in \delta\Lambda \cap \Omega_\delta\},
\]
with $e_2 \in \{\nu_1,\ldots,\nu_\ell\}$. Indeed, even though the stresses might not be in the $e_2$ direction exactly, any deviations are of order $1/L$, so it will not be optimal to change the angle from $\nu=e_2$ by a discrete amount.
\end{remark}

\subsection{Neumann boundary conditions} 

\label{sec:neumann}

In this case, we consider a bounded Lipschitz domain $O \subset \R^2$ that is symmetric in both axes, i.e.,
\[
O = \{x \in \R^2 \,|\, (x_1,-x_2) \in O\}=\{x \in \R^2 \,|\, (-x_1,x_2) \in O\},
\]
and consider $\Omega :=O \cap \R^{2}_{+}$ as our material domain. Dirichlet boundary conditions are imposed on $\Gamma:=\partial \Omega \cap \R^{2}_{+}$, while we have Neumann conditions on the remaining part $\partial \Omega \cap \partial \R^2_{+}$. For the activation sites, we consider the lattice $\Lambda_{+}$, cf.~\eqref{eq:latticeplus}, with zones entirely contained in $\Omega$ as in \eqref{eq:increments}.

The solution to \eqref{eq:displacementupdate} can be written
\[
w_{\nu,y}(x)=v_\nu(x-y)+\nu \cdot V_{y}(x) + v_{\nu}(x_1+y_1,x_2-y_2)+\nu \cdot V_{-y_1,y_2}(x),
\]
where $v_\nu$ is as in \eqref{eq:displcamentr2} and $V_y$ solves 
\begin{equation*}
\begin{cases}
\Delta V_{y} = 0 &\text{in $O$,}\\
V_{y}(x) = -V(x-y) &\text{on $\partial O$},
\end{cases}
\end{equation*}
in analogy  to \eqref{eq:diricorr}. From this, we find with \eqref{eq:quadraticcorr} that
\begin{equation}\label{eq:quadcorr2}
\Jcal_0(\nu \mathbbm{1}_{B_\delta(y)}) = \frac{\mu}{2}|B_\delta(0)|\left(\frac{1}{2}+\frac{\delta^2}{8y_1^2}(\nu_1^2-\nu_2^2)-\nu \cdot \nabla V_y(y)\nu - \nu \cdot \nabla V_{-y_1,y_2}(y)\nu\right),
\end{equation}
and, arguing as for \eqref{eq:enbound}, we find
\begin{equation}\label{eq:enboundneumann}
\begin{split}
&\frac{\mu}{2}|B_\delta(0)|\left(\frac{1}{2}+\frac{\delta^2}{8y_1^2}(\nu_1^2-\nu_2^2)\right)\\
&\qquad\qquad\leq \Jcal_0(\nu \mathbbm{1}_{B_\delta(y)}) \leq \frac{\mu}{2}|B_\delta(0)|\left(\frac{1}{2}+\frac{\delta^2}{8y_1^2}(\nu_1^2-\nu_2^2)+\frac{2\delta^2}{\dist(y,\partial O)^2}\right).
\end{split}
\end{equation}
We come to the main statement, whose proof follows almost exactly in the same way as Theorem~\ref{th:partialshearband} by reducing to the unbounded domain case Theorem~\ref{th:shearneumann}.
\begin{theorem}
\label{thm:bounded_neumann}
Let $O \subset \R^2$ be a bounded Lipschitz domain with $\delta e_1 \in O$ and such that 
\begin{equation}\label{eq:symmetry2}
O = \{x \in \R^2 \,|\, (x_1,-x_2) \in O\}=\{x \in \R^2 \,|\, (-x_1,x_2) \in O\},
\end{equation}
and $\Omega = O \cap \R^{2}_{+}$ and $\Gamma = \partial \O \cap \R^2_{+}$. Suppose the increments $\Ical$ are given by \eqref{eq:increments} with $\Lambda=\Lambda_{+}$ and consider the initial plasticity $p=\gamma_0 e_2 \mathbbm{1}_{B_\delta(\delta e_1)}$. If the boundary condition is given by $g(x)=\gamma x_2$ with
\begin{equation}\label{eq:stabilityinit2}
\gamma = \frac{\sigma_c}{\mu}+\gamma_0\left(\frac{\Jcal_0(e_2\mathbbm{1}_{B_\delta(3\delta e_1)})}{\mu|B_\delta(0)|}-\nabla w_{e_2,\delta e_1}(3\delta e_1)\cdot e_2\right) \geq \gamma_0,
\end{equation}
then, there exists a constant $\alpha >0$ independent of $\gamma_0,\gamma,\delta,O$ such that with $L:=\dist(\delta e_1,\partial O)/\delta$ it holds for $k \in \N$ that
\[
\{p\} \subsetneq \Acal_{k}(p) \subset \{p\} \cup \left\{\gamma_0\sum_{j=0}^{k} e_2\mathbbm{1}_{B_\delta((2j+1)\delta e_1)} \right\} \quad \text{if $k \leq \alpha L$}.
\]
\end{theorem}
\begin{proof}
By rescaling the domain, we may assume that $\delta=1$ and $L=\dist(e_1,\partial O)$.  Suppose now that $k \in \N$ and the plasticity is given by
\[
p' = \gamma_0\sum_{j=0}^{k-1} e_2\mathbbm{1}_{B_1(2(j+1) e_1)}.
\]
We infer from \eqref{eq:newdisplacementstress} that
  \begin{equation}\label{eq:9}
\frac{\sigma(p')(x)}{\mu}=\gamma e_2 +\gamma_0 \sum_{j=-k}^{k-1}\left( \nabla v(x-(2j+1)e_1) +\nabla V_{(2j+1)e_1}(x)e_2 - e_2\mathbbm{1}_{B_1((2j+1)e_1)}(x)\right)\,.
\end{equation}
If we denote the stress associated to the unbounded domain solution by
\[
\sigma_{\R^2_{+}}(p')(x) := \mu \left(\gamma e_2 + \gamma_0\sum_{j=-k}^{k-1} \nabla v(x-(2j+1)e_1) - e_2\mathbbm{1}_{B_1((2j+1)e_1)}(x)\right),
\]
then by Remark \ref{rem:unifestimate2} there exists $c >0$ such that
  \begin{equation}\label{eq:8}
|\sigma_{\R^2_{+}}(p')((2k+1)e_1)|+ \frac{\mu\gamma_0}{16(2k+1)^2} \geq |\sigma_{\R^2_{+}}(p')(x)|+\frac{\mu\gamma_0}{16x_1^2}+c\mu\gamma_0
\end{equation}
for all  $x \in \Lambda_{+} \setminus \{(2k+1)e_1\}$. Now, if $1 \leq k \leq L/4$, we infer from \eqref{eq:corrbound} with $\O$ replaced by $O$ and \eqref{eq:9} that
\begin{align}
\begin{split}\label{eq:stresscomp2}
\abs{\sigma(p')((2k+1)e_1)} &\geq \abs{\sigma_{\R^2_{+}}(p')((2k+1)e_1)}\\
&\qquad - \mu \gamma_0\sum_{j=-k}^{k-1} \frac{1}{\dist((2k+1)e_1,\partial O)\dist((2j+1)e_1,\partial O)} \\
&\geq \abs{\sigma_{\R^2_{+}}(p')((2k+1)e_1)} -  \frac{8\mu\gamma_0 k}{L^2},
\end{split}
\end{align}
where we assumed $\alpha<\frac14$, and hence 
\[
\min\{\dist((2k+1)e_1,\partial \Omega),\dist((2j+1)e_1,\partial \Omega)\} \geq L/2\,.
\]
If $\nu \in \S^1$, we find from \eqref{eq:changestz} and \eqref{eq:quadcorr2}
  \begin{equation}\label{eq:10}
    \begin{split}
\Fcal_{p'}(\gamma_0\nu \mathbbm{1}_{B_1((2k+1)e_1)}) &= -\gamma_0|B_1(0)| \left(\sigma(p')((2k+1)e_1) \cdot \nu -\sigma_c -\frac{\mu\gamma_0}{4}+\frac{\mu\gamma_0}{16(2k+1)^2}(\nu_2^2-\nu_1^2)\right.\\
 &\qquad\left. +\frac{\mu\gamma_0}{2}\nu \cdot \left(\nabla V_{(2k+1)e_1}((2k+1)e_1)+\nabla V_{-(2k+1)e_1}((2k+1)e_1)\right)\nu \right).
\end{split}
\end{equation}
By the symmetry of the domain \eqref{eq:symmetry2}, we find that $\nabla V_{(2k+1)e_1}((2k+1)e_1)$ and $\nabla V_{-(2k+1)e_1}((2k+1)e_1)$ are diagonal matrices with size bounded by $4/L^2$ by \eqref{eq:corrbound}. By symmetry of the domain and the covariance of the boundary conditions under the reflection $x\mapsto(-x_1,x_2)$,  the stress $\sigma(p')((2k+1)e_1)$ is a multiple of $e_2$. By the same arguments as in the proof of Theorem \ref{th:partialshearband}, we conclude that $\Fcal_{p'}(\nu \mathbbm{1}_{B_1((2k+1)e_1)})$ is minimized at $\nu = e_2$ for $k/L\leq \alpha$ sufficiently small. For $\nu=e_2$, we use  the estimates \eqref{eq:corrbound}, \eqref{eq:stresscomp2} in \eqref{eq:10} to obtain
\begin{equation}\label{eq:Fpxpm2}
\begin{split}
&\Fcal_{p'}(\gamma_0e_2 \mathbbm{1}_{B_1((2k+1)e_1)}) \\
&\qquad\leq -\gamma_0|B_1(0)| \left(|\sigma_{\R^2_{+}}(p')((2k+1)e_1)| -\frac{8\mu\gamma_0 k}{L^2}-\sigma_c -\frac{\mu\gamma_0}{4}+\frac{\mu\gamma_0}{16(2k+1)^2}-\frac{\mu\gamma_0}{L} \right)\,.
\end{split}
\end{equation}
This needs to be compared with $\mathcal F_{p'}(q_{\nu,x})$ for  $x \in \Lambda_{+} \cap \Omega_1 \setminus \{(2k+1)e_1\}$: In this case we have $\dist(x,\partial O) \geq 1$ so that \eqref{eq:corrbound} yields
  \begin{equation}\label{eq:11}
    \begin{split}
\abs{\sigma(p')(x)} &\leq \abs{\sigma_{\R^2_{+}}(p')(x)} + \mu\gamma_0 \sum_{j=-k}^{k-1} \frac{1}{\dist(x,\partial O)\dist((2j+1)e_1,\partial O)} \\
&\leq \abs{\sigma_{\R^2_{+}}(p')(x)} + \frac{4\mu\gamma_0 k}{L}\,.
\end{split}
\end{equation}
For   $x \in \Lambda_{+} \cap \Omega_1 \setminus \{(2k+1)e_1\}$, we may thus set up the following chain of inequalities:
\begin{align*}
\Fcal_{p'}(\gamma_0\nu \mathbbm{1}_{B_1(x)}) &\geq -\gamma_0|B_1(0)| \left(|\sigma_{\R^2_{+}}(p')(x)| -\sigma_c +\frac{4\mu\gamma_0 k}{L}- \frac{\mu\gamma_0}{4} +\frac{\mu\gamma_0}{16x_1^2} \right) \\
&\geq -\gamma_0|B_1(0)| \left(|\sigma_{\R^2_{+}}(p')((2k+1)e_1)| -c\mu\gamma_0-\sigma_c +\frac{4\mu\gamma_0 k}{L}- \frac{\mu\gamma_0}{4}+\frac{\mu\gamma_0}{16(2k+1)^2}  \right) \\
&>\Fcal_{p'}(\gamma_0 e_2 \mathbbm{1}_{B_1((2k+1)e_1)}),
\end{align*}
where the first inequality uses \eqref{eq:changestz}, \eqref{eq:enboundneumann} and \eqref{eq:11}, the second inequality uses \eqref{eq:8}, and the last inequality uses \eqref{eq:Fpxpm2} together with
\[
\frac{4 k}{L}+\frac{8k}{L^2}+\frac{1}{L}<c\,,
\]
which can be achieved by choosing $\alpha$ small enough.
This proves that either $\gamma_0e_2 \mathbbm{1}_{B_1((2k+1)e_1)}$ is the optimal non-zero plastic increment. 

\medskip

It remains to check that activating a zone at the endpoint is
energetically favorable when compared to not activating any plastic zone.
When $k=1$, we find from \eqref{eq:stabilityinit2} that
\[
\Fcal_{p'}(\gamma_0e_2\mathbbm{1}_{B_1(3e_1)})=0,
\]
so that both $0$ and
$\gamma_0e_2\mathbbm{1}_{B_1(3e_1)}$ are minimizing increments.

For $k>1$, we first determine the loading parameter more precisely. By
\eqref{eq:quadcorr2},
\[
\frac{\Jcal_0(e_2\mathbbm{1}_{B_1(3e_1)})}
     {\mu|B_1(0)|}
=
\frac12\left(\frac12-\frac1{72}\right)
-\frac12 e_2\cdot
\left(\nabla V_{3e_1}(3e_1)+\nabla V_{-3e_1}(3e_1)\right)e_2.
\]
Moreover, using the representation of $w_{e_2,e_1}$ and the identity
\[
\nabla v(re_1)\cdot e_2=\frac1{2r^2},
\]
we obtain
\[
\nabla w_{e_2,e_1}(3e_1)\cdot e_2
=
\frac18+\frac1{32}
+e_2\cdot
\left(\nabla V_{e_1}(3e_1)+\nabla V_{-e_1}(3e_1)\right)e_2.
\]
It follows from \eqref{eq:stabilityinit2} and \eqref{eq:corrbound} that
\[
\begin{aligned}
\gamma
&=
\frac{\sigma_c}{\mu}
+\gamma_0\left(\frac{35}{144}-\frac5{32}\right)
+O\left(\frac{\gamma_0}{L^2}\right)\\
&=
\frac{\sigma_c}{\mu}
+\frac{25\gamma_0}{288}
+O\left(\frac{\gamma_0}{L^2}\right).
\end{aligned}
\]

On the other hand, evaluating the explicit half-space stress
\eqref{eq:7} at the endpoint gives
\[
|\sigma_{\R^2_{+}}(p')((2k+1)e_1)|
=
\mu\left(
\gamma+\frac{\gamma_0}{8}
\sum_{j=1}^{2k}\frac1{j^2}
\right).
\]
Consequently,
\[
\begin{aligned}
&|\sigma_{\R^2_{+}}(p')((2k+1)e_1)|
+\frac{\mu\gamma_0}{16(2k+1)^2}
-\sigma_c-\frac{\mu\gamma_0}{4}\\
&\qquad=
\frac{\mu\gamma_0}{8}
\sum_{j=3}^{2k}\frac1{j^2}
+\frac{\mu\gamma_0}{16(2k+1)^2}
-\frac{\mu\gamma_0}{144}
+O\left(\frac{\mu\gamma_0}{L^2}\right).
\end{aligned}
\]
Here we have used that the corresponding leading-order expression vanishes
for $k=1$. If $k>1$, then
\[
\frac18\sum_{j=3}^{2k}\frac1{j^2}
+\frac1{16(2k+1)^2}
-\frac1{144}
\geq
\frac18\left(\frac19+\frac1{16}\right)-\frac1{144}
=
\frac{17}{1152}.
\]
It follows that
\[
|\sigma_{\R^2_{+}}(p')((2k+1)e_1)|
+\frac{\mu\gamma_0}{16(2k+1)^2}
-\sigma_c-\frac{\mu\gamma_0}{4}
\geq
\frac{17\mu\gamma_0}{1152}
-C\frac{\mu\gamma_0}{L^2}.
\]
Combining this with \eqref{eq:Fpxpm2}, we obtain
\[
\Fcal_{p'}(\gamma_0e_2\mathbbm{1}_{B_1((2k+1)e_1)})<0
\]
whenever $k/L\leq\alpha$, after decreasing $\alpha>0$ if necessary.
Thus the endpoint increment has strictly lower energy than the zero
increment for every $k>1$.
\end{proof}

\subsection*{Acknowledgements}
H.S. was funded by the Fonds de la Recherche Scientifique - FNRS through a MIS-Ulysse project (Scientific Impulse Mandate Instrument) number F.6002.25. 

\noindent The authors gratefully acknowledge Vasudevan Kamasamudram, Nicolas Mo\"es and Thomas Pardoen for their insightful comments and stimulating discussions during the development of this work.  

\subsection*{Generative AI declaration}
OpenAI’s ChatGPT was used as an interactive mathematical assistant during the preparation and revision of this article. Beyond general proof-auditing and expository improvements, including the creation of Figure~\ref{fig}, it was used in Section~\ref{sec:setting} to analyze the compactness and stability arguments and to scrutinize the passage from discrete minimizers to the stability conditions at limiting jumps. In Sections~\ref{sec:formation_thickening} and \ref{sec:half-space}, we used it to find initial formulas for some harmonic functions and infinite series. Finally, in Section~\ref{sec:bounded_domains}, it was used to check the perturbative comparison with the whole-space and half-space problems and to verify several explicit stress calculations, which led to some corrections of an early version of the manuscript. The mathematical ideas and results were developed by the authors; every AI-assisted argument and calculation was independently reconstructed and verified by the authors, who take full responsibility for the final manuscript.

\appendix

\section{Auxiliary series}

\label{sec:auxiliary_series}

\begin{lemma}\label{le:sums}
For every $a>0$,
\[
\sum_{j\in\mathbb Z}\frac{j^2-a^2}{(j^2+a^2)^2}
=
-\frac{\pi^2}{\sinh^2(\pi a)}\,,\qquad \sum_{j\in\mathbb Z}\frac{(2j+1)^2-a^2}{((2j+1)^2+a^2)^2}
=
\frac{\pi^2}{4\cosh^2(\pi a/2)}\,.
\]
\end{lemma}

\begin{proof}
The classical Mittag--Leffler expansion \cite[Chapter 7.4]{WhW20} 
 states that
  \begin{equation}\label{eq:ML}
\sum_{k\in\mathbb Z}\frac{1}{k^2+a^2}
=
\frac{\pi}{a}\coth(\pi a).
\end{equation}
Differentiating both sides with respect to $a$ and dividing by $-2a$ yields
\[
\sum_{k\in\mathbb Z}\frac{1}{(k^2+a^2)^2}
=
\frac{\pi}{2a^3}\coth(\pi a)
+
\frac{\pi^2}{2a^2}\operatorname{sinh}^{-2}(\pi a).
\]
Using the identity
\[
\frac{k^2-a^2}{(k^2+a^2)^2}
=
\frac{1}{k^2+a^2}
-
\frac{2a^2}{(k^2+a^2)^2},
\]
we obtain
\begin{align*}
\sum_{k\in\mathbb Z}\frac{k^2-a^2}{(k^2+a^2)^2}
&=
\frac{\pi}{a}\coth(\pi a)
-
2a^2
\left(
\frac{\pi}{2a^3}\coth(\pi a)
+
\frac{\pi^2}{2a^2}\operatorname{sinh}^{-2}(\pi a)
\right)\\
&=
-\pi^2\operatorname{sinh}^{-2}(\pi a)
=
-\frac{\pi^2}{\sinh^2(\pi a)},
\end{align*}
which proves the first identity.
For the second identity, let
\[
S(a):=\sum_{j\in\mathbb Z}\frac{1}{(2j+1)^2+a^2}=
\sum_{j\in\mathbb Z}\frac{1}{j^2+a^2}
-
\sum_{j\in\mathbb Z}\frac{1}{(2j)^2+a^2}\,.
\]
Since
\[
\sum_{j\in\mathbb Z}\frac{1}{(2j)^2+a^2}
=
\frac14
\sum_{j\in\mathbb Z}\frac{1}{j^2+(a/2)^2}
=
\frac{\pi}{2a}\coth\left(\frac{\pi a}{2}\right)\,,
\]
where we have used \eqref{eq:ML}, we find
\[
S(a)
=
\frac{\pi}{a}\coth(\pi a)
-
\frac{\pi}{2a}\coth\left(\frac{\pi a}{2}\right)
=
\frac{\pi}{2a}\tanh\left(\frac{\pi a}{2}\right)\,.
\]
Now
\[
\frac{(2j+1)^2-a^2}{((2j+1)^2+a^2)^2}
=
\frac{1}{(2j+1)^2+a^2}
-
\frac{2a^2}{((2j+1)^2+a^2)^2}.
\]
Since
\[
S'(a)
=
-2a\sum_{j\in\mathbb Z}
\frac{1}{((2j+1)^2+a^2)^2},
\]
we obtain
\[
\sum_{j\in \mathbb Z}
\frac{(2j+1)^2-a^2}{((2j+1)^2+a^2)^2}
=
S(a)+aS'(a)
=
\frac{d}{da}\bigl(aS(a)\bigr)= \frac{\pi^2}{4\cosh^2(\pi a/2)}\,.
\]
\end{proof}

\begin{lemma}
\label{lem:sumleq0}
Let $b\in\mathbb \R$ with $|b|\geq 1$, and define
\[
a_\ell:=\frac{\ell^2-b^2}{(\ell^2+b^2)^2}
\qquad\text{for }\ell\in\mathbb Z.
\]
Then for all $A,B\in\mathbb N\setminus\{0\}$,
\[
\sum_{\ell=-A}^{B} a_\ell \leq 0 \,,
\]
and for all $A\in \N$, $B\in\N\setminus\{0\}$, 
\[
\sum_{\ell=-A}^{B}a_{2\ell-1}\leq \frac{\pi^2}{4\cosh^2(\pi b/2)}\,.
\]
\end{lemma}

\begin{proof}
Since $a_\ell$ is even, we have
\[
\sum_{\ell=-A}^{B}a_\ell
=
\sum_{\ell=0}^{A}a_\ell+\sum_{\ell=0}^{B}a_\ell-a_0 .
\]
Set
\[
S_N:=\sum_{\ell=0}^{N}a_\ell .
\]
The terms $a_\ell$ are negative for $|\ell|<|b|$ and positive for $|\ell|>|b|$. Hence $S_N$ first decreases and then increases. Moreover, using Lemma \ref{le:sums},
\[
\sum_{\ell\in\mathbb Z}a_\ell
=
-\pi^2\sinh(\pi b)^{-2},
\]
we obtain
\[
\lim_{N\to\infty}S_N
=
\frac12\left(a_0+\sum_{\ell\in\mathbb Z}a_\ell\right)
=
-\frac12\left(\frac1{b^2}+\pi^2\sinh(\pi b)^{-2}\right)
\leq -\frac1{2b^2}.
\]
Since $S_0=a_0=-1/b^2$, it follows that
\[
S_N\leq -\frac1{2b^2}
\qquad\text{for all }N\geq0.
\]
Therefore
\[
\sum_{\ell=-A}^{B}a_\ell
=
S_A+S_B-a_0
\leq
-\frac1{2b^2}-\frac1{2b^2}+\frac1{b^2}
=0.
\]

\medskip

In the same way, we may set 
\[
\tilde S_B=\sum_{\ell=1}^B a_{2\ell-1}\,,
\]
and obtain 
\[
\sum_{\ell=-A}^B a_{2\ell-1}=\tilde S_{A+1}+\tilde S_B\,.
\]
Again $a_\ell$ will make a negative contribution for  $|2\ell-1|<|a|$, and a positive contribution for  $|2\ell-1|>a$. Hence 
\[
  \begin{split}
\tilde S_{A+1}+\tilde S_B&\leq \max\left\{2\tilde S_1,2\tilde S_\infty\right\}\\
&\leq \max \left\{\frac{2(1-b^2)}{(1+b^2)^2},\frac{\pi^2}{4\cosh^2(\pi b/2)}\right\}\\
&=\frac{\pi^2}{4\cosh^2(\pi b/2)}\,,
\end{split}
\]
where we have used Lemma \ref{le:sums}.
\end{proof}

\begin{lemma}\label{le:sums2}
For $k \in \N$ it holds that
\[
f(k):= \frac{1}{2}\sum_{j \in \Z} \frac{(k-2j)^2-3k^2}{((k-2j)^2+3k^2)^2} = \begin{cases}
\dfrac{\pi^2}{8\cosh^2\left(\frac{\pi k\sqrt{3}}{2}\right)} & \text{if $k$ is odd},\\[4ex]
\dfrac{-\pi^2}{8\sinh^2\left(\frac{\pi k\sqrt{3}}{2}\right)} & \text{if $k$ is even}.
\end{cases}
\]
Moreover, we have that
\begin{equation}\label{eq:monotone}
\abs{f(k)} > \abs{f(k+1)} \quad \text{and} \quad \abs{f(k)+f(k+1)} > \abs{f(k+1)+f(k+2)} \quad \text{for all $k \in \N$.}
\end{equation}
\end{lemma}
\begin{proof}
The formula for $f(k)$ follows directly from Lemma~\ref{le:sums}. Next, the first property of \eqref{eq:monotone} follows from the identities
\[
\cosh(\pi (s+1)\sqrt{3}/2) = \cosh(\pi s\sqrt{3}/2)\cosh(\pi\sqrt{3}/2)+\sinh(\pi s\sqrt{3}/2)\sinh(\pi\sqrt{3}/2)>\sinh(\pi s\sqrt{3}/2)
\]
and
\[
\sinh(\pi (s+1)\sqrt{3}/2) = \sinh(\pi s\sqrt{3}/2)\cosh(\pi\sqrt{3}/2)+\cosh(\pi s\sqrt{3}/2)\sinh(\pi\sqrt{3}/2)>\cosh(\pi s\sqrt{3}/2),
\]
where we used that $\sinh(\pi\sqrt{3}/2)>1$. The second property follow from the fact that
\[
x \mapsto \frac{1}{\cosh^2(x)}-\frac{1}{\sinh^2(x+\pi\sqrt{3}/2)} \quad \text{and} \quad x\mapsto \frac{1}{\sinh^2(x)}-\frac{1}{\cosh^2(x+\pi\sqrt{3}/2)}
\]
are decreasing on $[1,\infty)$.
\end{proof}
\bibliographystyle{abbrv}
\bibliography{bibliography}
\end{document}